\documentclass[12pt]{article}
	\usepackage{amsmath,fullpage,amsthm,amssymb,xcolor,graphicx}
	
	\newtheorem{theorem}{Theorem}[section]
	\newtheorem{corollary}{Corollary}[section]
	\newtheorem{lemma}{Lemma}[section]
	\newtheorem{definition}{Definition}[section]
	\newtheorem{remark}{Remark}[section]
	\newtheorem{example}{Example}[section]
	
	\newtheorem{prob}{Problem}[section]
	
	\DeclareMathOperator{\tr}{trace}
	\DeclareMathOperator{\spec}{spec}
		\DeclareMathOperator{\diag}{diag}
				\DeclareMathOperator{\rank}{rank}

       \renewcommand{\Re}{\operatorname{Re}}

	\title{Bounds for the energy of a complex unit gain graph\footnote{This paper is dedicated  to Professor  Ravindra Bhalchandra Bapat on the occasion of his 65th birthday with much admiration. }}
\author{Aniruddha Samanta \thanks{Department of Mathematics, Indian Institute of Technology Kharagpur, Kharagpur 721302, India. Email: aniruddha.sam@gmail.com}\  \and M. Rajesh Kannan\thanks{Department of Mathematics, Indian Institute of Technology Kharagpur, Kharagpur 721302, India. Email: rajeshkannan@maths.iitkgp.ac.in, rajeshkannan1.m@gmail.com }
}
\date{\today}
\begin{document}
\maketitle
\baselineskip=0.25in

\begin{abstract}
	A $\mathbb{T}$-gain graph, $\Phi = (G, \varphi)$, is a  graph in which the function $\varphi$ assigns a unit complex number to each orientation of an edge, and its inverse is assigned to the opposite orientation.  The associated adjacency matrix $ A(\Phi) $ is defined canonically. The energy $ \mathcal{E}(\Phi) $ of a $ \mathbb{T} $-gain graph  $ \Phi $ is the sum of the absolute values of all eigenvalues of $ A(\Phi) $.
	We study the notion of energy of a vertex of a $ \mathbb{T} $-gain graph, and establish bounds for it. For any $ \mathbb{T} $-gain graph $ \Phi$, we prove that $2\tau(G)-2c(G) \leq \mathcal{E}(\Phi) \leq 2\tau(G)\sqrt{\Delta(G)}$, where $ \tau(G), c(G)$ and $ \Delta(G)$ are the vertex cover number, the number of odd cycles and the largest vertex degree of $ G $, respectively. Furthermore, using the properties of vertex energy, we characterize the classes of  $ \mathbb{T} $-gain graphs for which   $ \mathcal{E}(\Phi)=2\tau(G)-2c(G) $ holds. Also,
	we characterize the classes of  $ \mathbb{T} $-gain graphs for which  $\mathcal{E}(\Phi)= 2\tau(G)\sqrt{\Delta(G)} $ holds.  This characterization solves a general version of an open problem. In addition, we establish  bounds for the energy  in terms of the spectral radius of the associated adjacency matrix.
\end{abstract}

{\bf AMS Subject Classification(2010):} 05C50, 05C22, 05C35.
\section{Introduction}
In a simple undirected graph  $ G $ with vertex set $ V(G)=\{  v_1, \dots, v_n\} $ and edge set $ E(G) $,  if two vertices $ v_p $ and  $ v_q $ are adjacent in $G$, then we write $ v_p \sim v_q $, and the edge in between them is denoted by $ e_{p,q} $. The number of vertices adjacent with the vertex $ v_p $, the \emph{degree} of $ v_p $, is denoted by $ d(v_p) $ (or simply $ d_p $).  $ \Delta(G) $ denotes the maximum vertex degree of $ G $.
% For undirected edge $ e_{i,j} $, it is clear that $ e_{i,j}=e_{j,i} $.
   A \emph{directed graph(or digraph)} $ X $ is an order pair $ (V(X), E(X)) $, where\break $ V(X)=\{ v_{1}, v_{2}, \dots,v_{n}\} $ is the vertex set and $ E(X) $ is the directed edge set. A directed edge from the vertex $ v_{p} $ to the vertex $ v_{q} $ is denoted by $ \overrightarrow{e_{p,q}} $. If $ \overrightarrow{e_{p,q}} \in E(X)$ and $  \overrightarrow{e_{p,q}}\in E(X)$, then the pair $ \{v_{p},v_{q}\} $ is called a \emph{digon} of $ X $. The  \textit{Hermitian adjacency matrix} \cite{Bojan, Lie} of a digraph $ X $ is  denoted by $H(X)$ and is defined as follows:
    $$\mbox{$ (p,q){th}$ entry of }H(X)=h_{p,q}=\begin{cases}
    1& \text{if both } \mbox{$\overrightarrow{e_{p,q}}$ \text{and} $\overrightarrow{e_{q,p}} \in E(X)$},\\
    i& \text{if  } \mbox{$\overrightarrow{e_{p,q}} \in E(X)$ \text{and} $\overrightarrow{e_{q,p}} \notin E(X)$},\\
    -i& \text{if  } \mbox{$\overrightarrow{e_{p,q}} \notin E(X)$ \text{and} $\overrightarrow{e_{q,p}} \in E(X)$},\\
    0&\text{otherwise.}\end{cases}$$
        The Hermitian adjacency matrix can be thought of as the adjacency matrix of a $\mathbb{T}$-gain graph with the gains are from the set $ \{1, \pm i\}$. A digraph is said to be an \emph{oriented graph} if it has no digons. A graph contains both directed and undirected edges is called a \emph{mixed graph} and it is denoted by $ D_G $, where $ G $ is the underlying simple graph. When we consider Hermitian adjacency matrix, $ H(D_G) $  of a mixed graph $ D_{G} $, the undirected edges are treated as digons.

 From a simple graph $G$, by orienting each undirected edge $ e_{p,q} \in E(G)$ in two opposite directions, namely $ \overrightarrow{ e_{p,q}}$ and $ \overrightarrow{e_{q,p}}$, we get a digraph.  Let $ \overrightarrow{E(G)}=\{\overrightarrow{ e_{p,q}}, \overrightarrow{e_{q,p}}: e_{p,q}\in E(G) \} $ and $ \mathbb{T}=\{ z \in \mathbb{C}: |z|=1\} $. A \emph{complex unit gain graph} (simply, $ \mathbb{T} $-gain graph) on a simple graph $ G $ is a pair $ (G, \varphi) $, where $ \varphi: \overrightarrow{E(G)} \rightarrow \mathbb{T} $ is a mapping  such that $ \varphi( \overrightarrow{e_{p,q}}) =\varphi(\overrightarrow{e_{q,p}})^{-1}$. A $ \mathbb{T} $-gain graph $ (G, \varphi) $ is  denoted by $ \Phi $.  For more details about the $\mathbb{T}$-gain graphs, we refer to \cite{Our-gain1, reff1,Reff2016, Our-paper-2, Zas4}.

The \emph{ adjacency matrix} of $ \Phi$ is the Hermitian  matrix $ A(\Phi)=(a_{p,q})_{n \times n }$  defined as follows:
    $$a_{p,q}=\begin{cases}
    \varphi(\overrightarrow{e_{p,q}})&\text{if } \mbox{$v_p\sim v_q$},\\
    0&\text{otherwise.}\end{cases}$$

    Let $ \{\lambda_1, \dots, \lambda_n\} $ be the spectrum of $ A(\Phi) $ (or the spectrum of $ \Phi $), and is denoted by $\spec(\Phi)$. The energy of $ \Phi $, denoted by  $ \mathcal{E}(\Phi) $, is defined by $\sum\limits_{j=1}^{n}|\lambda_j| $.

    For a vertex $ v_j $ of $ G $, the \emph{energy of the vertex $ v_j $}, denoted by $ \mathcal{E}_{G}(v_j) $, is defined by $ \mathcal{E}_{G}(v_j)=|A(G)|_{jj} $, where $ |A(G)|=( A(G)A(G)^{*} )^{\frac{1}{2}}$ and $ |A(G)|_{jj} $ is the $ (j,j)$-th entry of $ |A(G)| $. Then $ \mathcal{E}(G)=\sum\limits_{j=1}^{n}\mathcal{E}_{G}(v_j) $ \cite{Vertexenergy1}. In Section \ref{ver-ener}, we establish bounds for $\mathcal{E}_{\Phi}(v_j) $, the vertex energy of  a $\mathbb{T}$-gain graph, in terms degree of the vertex $v_j$, and characterize the classes of graphs for which the bounds are sharp. As a consequence of these bounds, we provide a couple of bounds for the energy of a $\mathbb{T}$-gain graph in terms of the energy of the underlying graph and the number of vertices of the graph.

    A \emph{matching} in a graph $ G $ is a set of edges of $ G $ such that no two edges are incident with the same vertex. The cardinality of a matching with the maximum number of edges is the \emph{matching number} of $ G $, and is denoted by $ \mu(G) $. A matching that saturates all the vertices of $ G $ is known as a \emph{perfect matching} of $ G $. A \emph{vertex cover} $ U $ of a graph $ G $ is a subset of $ V(G) $ such that every edge  of $ G $ is incident with at least one vertex of $ U $. The cardinality of a vertex cover with the minimum number of vertices is the \emph{vertex cover number} of $ G $, and is denoted by $ \tau(G)$. For any $ \mathbb{T} $-gain graph $ \Phi=(G, \varphi) $, the matching number, and the vertex cover number of $ \Phi $ are the matching number and the vertex cover number of the underlying graph $ G $, respectively.

     In \cite{Wang-Ma}, the authors derived a lower bound for the  energy of an undirected graph in terms of the vertex cover number and the number of odd cycles.
    \begin{theorem}[{\cite[Theorem 4.2]{Wang-Ma}}]\label{lm.15}
    If $ G $ is a graph with the vertex cover number $ \tau(G) $ and the number of odd cycle $ c(G) $, then $ \mathcal{E}(G)\geq 2\tau(G)-2c(G) $. Equality occurs if and only if each component of $ G $ is a complete bipartite graph with perfect matching together with some isolated vertices.
    \end{theorem}

     In \cite{Wei-Li}, the authors extended  Theorem \ref{lm.15} for Hermitian adjacency matrices of mixed graphs.

     \begin{theorem}[{\cite[Theorem 4.5]{Wei-Li}}] \label{th2}
      Let $ D_G $ be a mixed graph with vertex cover number $ \tau(G) $ and number of odd cycles $ c(G) $. Then $ \mathcal{E}_{H}(D_G) \geq 2\tau(G)-2c(G) $. Equality occurs if and only if $ D_G $ is switching equivalent to its underlying graph $ G $, where each component of $ G $ is either a complete bipartite graph with equal partition size or isolated vertices.
     \end{theorem}
     Further extensions of  Theorem \ref{lm.15}  are discussed in \cite{Tian-Wong, Wong-Wang-Chu}.

     In Section \ref{lower-bound}, we obtain  lower bounds for  $ \mathcal{E}(\Phi) $ in terms  of the gains of fundamental cycles [Theorem \ref{th-4.1} and Theorem \ref{lm2}]. We show that a connected $\mathbb{T}$-gain bipartite graph has exactly one positive eigenvalue if and only if it is the balanced complete bipartite graph [Theorem \ref{lem5}].  We establish a bound for the energy of a $\mathbb{T}$-gain graph in terms of the spectral radius of $ \Phi $, and characterize the sharpness of the inequality [Theorem \ref{Th6}]. Further, we establish lower bounds for  $ \mathcal{E}(\Phi) $ in terms of the vertex cover number, the number of odd cycles, and the matching number [Theorem \ref{th4(ii)} and Theorem \ref{th8}]. After completion of this work, we learned that Theorem \ref{th4(ii)} has been proved in \cite{gain-ver-cov-new} independently. However, our proof uses the properties of vertex energy of $\mathbb{T}$-gain graphs, and hence the proof is different from the proof given in \cite{gain-ver-cov-new}.

      In \cite{Wang-Ma}, the authors established an upper bound of the energy of an undirected graph in terms of the vertex cover number and the largest vertex degree.

         \begin{theorem}
         \cite[Theorem 3.1]{Wang-Ma}\label{th-1.3} If $ G $ is an undirected
         graph with vertex cover number $ \tau(G) $ and maximum vertex degree $
         \Delta(G)$, then $ \mathcal{E}(G)\leq 2\tau(G)\sqrt{\Delta(G)} $. Equality occurs if and only if $ G $ is the disjoint union of $ \tau(G) $ copies of $ K_{1, \Delta(G)} $ together with some isolated vertices.
         \end{theorem}

         In \cite{Wei-Li}, the authors extended this inequality for a mixed graph and proposed the equality part as an open problem.

         \begin{theorem}    [{\cite[Theorem 4.9]{Wei-Li}}] \label{th-1.4}
         Let $ D_G $ be a mixed graph with vertex cover number $ \tau(G) $ and largest vertex degree $
         \Delta(G) $. Then
         \begin{equation}\label{eq1}
         \mathcal{E}_{H}(D_G)\leq 2\tau(G)\sqrt{\Delta(G)}.
         \end{equation}
         \end{theorem}

         In Section \ref{upper-bound}, we extend Theorem \ref{th-1.4} for the $\mathbb{T}$-gain graphs [Theorem \ref{th-5.3}].

         \begin{prob}[{\cite[Problem 4.1]{Wei-Li}}]\label{problem1}
          Characterize all mixed graphs which make the equality in \eqref{eq1}  hold.
         \end{prob}

         We solve this  problem for the $\mathbb{T}$-gain graphs [Theorem \ref{th-5.4}]. The Hermitian adjacency matrices of mixed graphs are particular cases of the adjacency matrices of the $\mathbb{T}$-gain graphs. Also, in a recent manuscript \cite{gain-ver-cov-new}, the author mentioned the difficulties in extending  Theorem \ref{th-1.4}, and characterizing the graphs for which equality hold for the $\mathbb{T}$-gain graphs.

         This article is organized as follows: In Section \ref{prelim}, we collect needed known definitions and results. In Section \ref{ver-ener}, we extend the notion of vertex energy for $\mathbb{T}$-gain graphs, and establish some of the properties.  In Section \ref{lower-bound}, we establish various lower bounds for the energy of $\mathbb{T}$-gain graphs, and  Section \ref{upper-bound} is devoted to upper bounds for the energy of $\mathbb{T}$-gain graphs.

	%===============================================	  	
\section{Definitions, notation and preliminary results}\label{prelim}

In this section we recall some of the needed graph theory and linear algebra terminologies and some of the basic results.
A subgraph $ H $ of a graph $ G $ is  an \emph{induced subgraph} if  two vertices of $ H $ are adjacent in $ G $, then they are adjacent in $ H $. For an induced subgraph $ H $ of $G$ the \emph{complement} of $ H $  in $ G $, denoted by $G-H$,  defined as the induced subgraph of $ G $ with vertex set $ V(G)\setminus V(H) $. The subgraphs  $ H $ and $ G-H $ are called \emph{complementary induced subgraphs} in $ G $. If $ E $ is any edge set of $ G $, then $ G-E $ denotes the spanning subgraph of $ G $ with edge set $ E(G)\setminus E $ and vertex set $ V(G) $. A \emph{cut} of a graph $ G $ is a partition of the vertex set $ V(G) $ into two sets $ U $ and $ W $. A \emph{cut set} of $ G $ is a set of edges $ \{e_{p,q}\in E(G): v_p \in U, v_q \in W\} $, where $U$ and $V$ partition the vertex set $V(G)$. Suppose $ E $ is a cut set, then there are two induced subgraphs $ H $ and $ G-H $ complement to each other such that each edge of $ E $ is incident to a vertex of  $ H $ and to another vertex of $ G-H $ \cite{Day-So}. Then we denote $ H \oplus (G-H)=G-E $.

Let $ e_{p,q} \in E(G) $. To avoid confusion, we denote $ G-[e_{p,q}] $ as an induced subgraph of $ G $ whose vertex set is $ V(G)\setminus\{v_p, v_q\} $.  If $ K$ is a spanning subgraph of $ G $, then for any edge $ e\in E(G)\setminus E(K) $, $ K+e $ denotes a spanning subgraph of $ G $ with the edge set $ E(K)\cup\{e\} $. If   $ G $ is a connected graph  and $T$ is a spanning tree  of $ G $, then  any edge $ e\in E(G)\setminus E(T) $ induces a unique cycle in $ T+e $.  This is called a \emph{fundamental cycle} in $ G $ with respect to $ T $.

The \textit{adjacency matrix} of a simple graph $ G $, denoted by $ A(G) $,  is the symmetric  $ n \times n $  matrix whose  $ (p,q)th $ entry  is defined by $ a_{p,q}=1 $ if $ v_{p} \sim v_{q} $, and  $ a_{p,q} = 0 $ otherwise. The energy of the graph $G$, denoted by $\mathcal{E}(G)$, is the sum of the absolute values of  the eigenvalues of $A(G)$.

\begin{lemma}[{\cite[Theorem 3.6]{Day-So}}]\label{lm.13}
Let $ L $ and $ M $ be two complementary induced subgraph of a graph $ G $ and $ E $ be the cut set in between them. If $E$ is not empty and all edges in $E$ are incident to one and only one vertex in $ M $, then $ \mathcal{E}(G-E) < \mathcal{E}(G)$.
\end{lemma}

Let $ \Phi=(G, \varphi) $ be any $ \mathbb{T} $-gain graph, and $ H $ be a subgraph of $ G $. We call  $(H , \xi)$  a subgraph of $\Phi$ if the  function $\xi$ is the restriction of $ \varphi $ on $\overrightarrow{E(H)}$, and is denoted by $ (H, \varphi) $ (instead of $(H , \xi)$). If $ H $ is an induced subgraph of $ G $ and $ E $ is any edge set of $ G $, then similar to undirected graphs we can define $ \Phi-H $ and $ \Phi-E $.

The \emph{adjacency matrix} of $ \Phi=(G, \varphi) $, denoted by $ A(\Phi) $, is defined as the Hermitian matrix  whose $ (p,q)$-th element is $ \varphi(\overrightarrow{e_{p,q}}) $ if $ v_p \sim v_q $ and, zero otherwise. The spectrum of $ \Phi $, denoted  by $ \spec(\Phi) $, is  the spectrum of $ A(\Phi) $. The spectral radius of $ \Phi $ is denoted by $ \rho(\Phi) $.	 The
 \textit{energy} of $ \Phi $, denoted by $ \mathcal{E}(\Phi) $,  is defined as $ \mathcal{E}(\Phi)=\sum_{j=1}^{n}|\lambda_j| $, where $\lambda_j$ are the eigenvalues of $\Phi$.
Two $ \mathbb{T} $-gain graphs $ \Phi=(G, \varphi) $ and $ \Phi^{'}=(G, \varphi^{'}) $ are  \emph{switching equivalent} if there exists a unitary diagonal matrix $ U $ such that $ A(\Phi^{'})=UA(\Phi)U^{*} $. If $ \Phi $ and $ \Phi^{'} $ are switching equivalent, then it is denoted by $ \Phi \sim \Phi^{'} $.

 A directed cycle  is called an \emph{oriented cycle} if all of its edges are directed such that each edge is traversed in the same direction. An undirected cycle of $ k $ vertices $ C \equiv v_1-v_2-\dots-v_k-v_1 $   has two oriented cycles. If one of  the orientation, say $ v_1 \rightarrow v_2\rightarrow\dots v_k \rightarrow v_1 $, is denoted by $ \overrightarrow{C} $, then opposite oriented cycle is denoted by $ \overrightarrow{C}^{*} $. The gain of an oriented cycle $ \overrightarrow{C} $ is defined as $ \varphi(\overrightarrow{C})=\varphi(\overrightarrow{e_{1,2}})\varphi(\overrightarrow{e_{2,3}})\cdots \varphi(\overrightarrow{e_{k,1}}) $. Therefore, $ \varphi(\overrightarrow{C}^{*})=\left\{\varphi(\overrightarrow{C}) \right\}^{-1} $.  For any complex number $ \lambda$, $ \Re(\lambda) $ denotes the real part of $ \lambda $. If $ \varphi(\overrightarrow{C}) = \varphi(\overrightarrow{C}^{*})=1$, then we simply write $ \varphi(C)=1 $. Similarly, for any cycle $ C $, $ \Re (\varphi(\overrightarrow{C}))=\Re (\varphi(\overrightarrow{C}^{*})) $. Thus, we simple write $ \Re(\varphi(C)) $.

 A $ \mathbb{T} $-gain graph $ \Phi=(G, \varphi) $ is called \emph{balanced} if $ \varphi(\overrightarrow{C}) =1$, for any cycle $ C $ in $ G $.  If $ \Phi $ is balanced, then $ \Phi \sim (G,1) $. Some of the properties of $\mathbb{T}$-gain graphs are collected in the next couple of results.
\begin{theorem}[{ \cite[Lemma 4.1, Theorem 4.4]{Our-gain1}}] \label{th-2.1}
Let $ \Phi=(G, \varphi) $ be any $ \mathbb{T} $-gain graph on a connected graph $ G $. Then $ \rho(\Phi) \leq \rho(G)$. Equality occur if and only if either $ \Phi $ or $ -\Phi $ is balanced.
\end{theorem}
	
 \begin{theorem}[{\cite[Theorem 4.5]{Our-gain1} }]\label{th-2.2}
 Let $G$ be a connected graph. Then we have the following:
 \begin{enumerate}
  \item If $ G $ is bipartite, then whenever $ \Phi $ is balanced implies $ -\Phi $ is balanced.
  \item If $ \Phi $ is balanced implies $ -\Phi $ is balanced for some gain, then $ G $ is bipartite.
 \end{enumerate}
 \end{theorem}

 \begin{lemma}[{\cite[Corollary 3.2]{Our-paper-2}}]\label{lm-2.1}
 Let $ \Phi_1=(G, \varphi_1) $ and $ \Phi_2=(G, \varphi_2) $ be two $\mathbb{T}$-gain graphs on a connected graph $ G $ with $ n $ vertices and $ m $ edges. Let $ \{ C_1, C_2, \dots, C_{m-n+1}\} $ be the fundamental cycles of $ G $ with respect to a normal spanning tree of $ G $. Then $ \Phi_1 \sim \Phi_2 $ if and only if $ \varphi_1(\overrightarrow{C_j}) = \varphi_2(\overrightarrow{C_j})$, for all $ j=1, 2, \dots, (m-n+1)$.
 \end{lemma}

Let $C_n$ denote the cycle on  $n$ vertices.
\begin{theorem}[{\cite[Theorem 6.1]{reff1}}] \label{th-2.3}
 Let $ \Phi=(C_n, \varphi) $ be a $ \mathbb{T} $-gain graph with $ \varphi(\overrightarrow{C_{n}})=e^{i\theta} $. Then
\begin{equation}
\spec(\Phi)=\left\{2\cos \left( \frac{\theta+2\pi j}{n}\right): j=0,1, \dots, (n-1) \right\}.
\end{equation}
\end{theorem}

\begin{lemma}[{ \cite[Theorem 1.13]{He-Xia-Dong}}]\label{2-lm5}
Let $ \Phi=(G, \varphi) $ be any connected $ \mathbb{T} $-gain graph. Then $$2\max\limits_{V_0}\mu(G-V_0)\leq r(G, \varphi)\leq 2\mu(G)+b(G),$$
where $ V_0 $ is any proper subset of $ V(G) $ such that $ G-V_0 $ is acyclic and $ b(G) $ is the minimum integer $ |U| $ such that $ G-U $ is bipartite, $ U \subset V(G) $.
\end{lemma}

 In \cite{Vertexenergy1}, the authors studied the notion of vertex energy of a graph.
\begin{definition} [{\cite[Definition 2.1]{Vertexenergy1}}]
Let $ G $ be a graph with vertex set $ V(G)=\{ v_1, v_2, \dots, v_n\} $. Then the \emph{energy of a vertex} $ v_j $, denoted by $ \mathcal{E}_{G}(v_j) $, is defined as $\mathcal{E}_{G}(v_j)=|A(G)|_{jj} $, where $ |A(G)|=(A(G)A(G)^{*})^{\frac{1}{2}} $.
\end{definition}

Next,  we recall a few results related to the vertex energy.
\begin{lemma}[{\cite[Lemma 2.2]{Vertexenergy1}}]\label{lm-2.2}
 Let $ G $ be an undirected graph with vertex set $ V(G)=\{ v_1, v_2, \dots, v_n\} $. Then
\begin{equation}
\mathcal{E}_{G}(v_i)=\sum\limits_{j=1}^{n}Q_{ij}|\lambda_j|, \text{  for } i=1,2, \dots n.
\end{equation}
where $ Q_{ij}=q_{ij}^{2} $ and $ Q=(q_{ij}) $ is the orthogonal matrix whose columns are the eigenvectors of $ G $ and $ \lambda_j $ is the $ j$-th  eigenvalue of $ G $.	  	
\end{lemma}

\begin{lemma}[{\cite[Theorem 3.3]{Vertexenergy1}}] \label{th3}
If $ G $ is a connected graph on $n$ vertices with at least one edge, then
\begin{equation}
\mathcal{E}_{G}(v_j) \geq \frac{d_j}{\Delta(G)},\qquad \text{ for all $ v_j \in V(G) $}.
\end{equation}
Equality occurs if and only if $ G $ is a complete bipartite graph with equal partition size.
\end{lemma}

Let $ G $ and $ G_1 $ be two simple graphs. Let  $ D_G $ be a mixed graph on $ G $. The mixed Kronecker product, denoted by  $ D_{G}\otimes G_1 $,  is the Kronecker product of the Hermitian adjacency matrix of $G$ and the adjacency matrix of the simple graph $G_1$ \cite{Wei-Li}.

\begin{lemma}[{\cite[ Lemma 2.7]{Wei-Li}}]\label{lm.12} Let $ \{ \lambda_1, \lambda_2, \dots, \lambda_s \} $ be the spectrum of $ G_1 $, and $ \{\gamma_1, \gamma_2, \dots, \gamma_t \} $ be the spectrum of $ D_G $(with respect to the Hermitian adjacency matrix), then the spectrum of a mixed Kronecker product $ D_G \otimes G_1 $ is $ \{\lambda_i\gamma_j: 1 \leq i \leq s, 1 \leq j\leq t \} $
\end{lemma}

 The \emph{Hermitian energy} of a mixed graph $ D_G $ is the sum of the absolute values of the eigenvalues of $ H(D_G) $, and is denoted by $ \mathcal{E}_{H}(D_G) $.

 Let us collect a few results  on energy in terms of matching number.

\begin{lemma}[{\cite[Lemma 4.1]{Wang-Ma}}] \label{lm.14}
 For any bipartite graph $ G $, $ \mathcal{E}(G)\geq 2\mu(G) $. Equality occur if and only if each component of $ G $ is complete bipartite graph with perfect matching together with some isolated vertices.
\end{lemma}
	
\begin{theorem}[{\cite[Theorem 1.1]{Wong-Wang-Chu}}] \label{lm.16}
Let $ G $ be a graph with matching number $ \mu(G) $. Then $ \mathcal{E}(G)\geq 2\mu(G) $. If all cycles (if any) of $ G $ are pairwise vertex disjoint, then equality holds if and only if each component of $ G $ is either an edge or $ 4 $-cycle or an isolated vertices.
\end{theorem}
	
\begin{theorem}[{\cite[Theorem1.1, Theorem 1.2]{Wei-Li}}]\label{lm.18}
Let $ D_G $ be a mixed graph with matching number $ \mu(G) $, then $ \mathcal{E}_{H}(D_G) \geq 2\mu(G) $. Equality occur if and only if $ D_G $ is switching equivalent to its underlying graph $ G $, where each component of $ G $ is either a complete bipartite graph with equal partition size or isolated vertices.
\end{theorem}

 \begin{lemma}[{\cite[Lemma 3.8]{Wei-Li}}]\label{lm.11}
 Let $ D_G $ be a mixed graph on a connected non bipartite graph $ G $. Then $ \mathcal{E}_{H}(D_G) > 2 \mu(G) $.
 \end{lemma}

 \begin{lemma}[{\cite[Lemma 3.6]{Wei-Li}}]\label{lm14}
Let $ D_G $ be a mixed graph without isolated vertices. If $\mathcal{E}_{H}(D_G)=2\mu(G)  $, then $ G $ has a perfect matching.
\end{lemma}

A graph $ G $ is  \emph{bipartite graph} if its vertex set $ V(G) $ can be partitioned into two sets, $ X $ and $ Y $ such that every edge of $ G $  joins a vertex of $ X $ with a  vertex of $ Y $. If every vertex in $ X $ is adjacent to every vertex in $ Y $, then the graph $ G $ is called a \emph{complete bipartite graph}. If $ G $ is a complete bipartite graph with $ |X|=p $ and $ |Y|=q $, then $ G $ is denoted by $ K_{p,q} $. For instance, $ K_{p,p} $ is a complete bipartite graph with a perfect matching. A graph $ G $ is called an \emph{ $ r $-regular graph ( or regular graph )} if every vertex of $G $ has the same degree $ r $. A graph $ G $ is called a \emph{semiregular bipartite graph } with parameter $ (n_a, n_b, r_a, r_b) $ if $ G $ is a bipartite graph with $|X|=n_a $ and $ |Y|=n_b $ such that all the vertices of $ X $ have the same degree $ r_a $ ,and the vertices of $ Y $ have the same degree $ r_b $.

\begin{theorem}[{\cite[Theorem 3]{Gutman-Zare-Pena-Rada}}] \label{lm.19}
 If $ G $ is a $ d $-regular graph of $ n $ vertices, then $ \mathcal{E}(G)\geq n $. Equality holds if and only if each component is isomorphic to $ K_{d,d} $.
\end{theorem}

Let $ G $ be a semiregular bipartite graph with partition size  $n_a $ and $ n_b $, and  the vertex degree of each vertex of first and second partition is  $ r_a$ and $ r_b $, respectively. The next result provides a bound of  $ \mathcal{E}(G).$
\begin{theorem}[{\cite[Theorem 5]{Gutman-Zare-Pena-Rada}}]\label{lm.20}
If $ G $ is a semiregular graph with the parameter $ (n_a, n_b, r_a, r_b ) $.  Then
$ \mathcal{E}(G)\geq n_a \sqrt{\frac{r_a}{r_b}}+n_b\sqrt{\frac{r_b}{r_a}} $ and equality occur if and only if every component of  $ G $ is $ K_{r_{a}, r_{b}} $.
\end{theorem}

For an $n \times n $ complex square matrix $A$, $\tr(A)$ denotes the trace of the matrix $A$. The next result is known as the von Neumann's trace theorem.
\begin{theorem}[{\cite{horn-john2}}] \label{th1}
Let $ A $ and $ B $ be two square complex matrices with singular values $ \lambda_1(A)\geq \lambda_2(A)\geq \dots \geq  \lambda_n(A) $ and $\lambda_1(B)\geq \lambda_2(B)\geq \dots \geq \lambda_n(B) $, respectively. Then
\begin{equation}
\Re(\tr(AB))\leq \sum\limits_{j=1}^{n}\lambda_{j}(A)\lambda_{j}(B).
\end{equation}
\end{theorem}

\begin{theorem}[{\cite[Corollary 2.4.]{Day-So}}]\label{lm12}
If $ A=\left[ \begin{array}{cc}
B & X\\ Y& C \end{array} \right]$ is any partition matrix with $ A $ and $ B $ are the square matrices, then $ \mathcal{E}(A)\geq \mathcal{E}(B) $. Equality occurs if and only if $ X, Y $ and $ C $ are all zero matrices.
\end{theorem}

\begin{theorem}[{\cite[Theorem 2.2]{Day-So}}]\label{lm11}
 Let $ A= \left [\begin{array}{cc}
 A_{11} & A_{12}\\ A_{21} & A_{22}
\end{array}\right] $ be a complex block matrix such that both the diagonal blocks are square matrices. Then $ \mathcal{E}(A_{11})+\mathcal{E}(A_{22}) \leq \mathcal{E}(A)  $, where $ \mathcal{E}(A) $ is the sum of the singular values of $ A $ . Equality occurs if and only if there exist unitary matrices $ U $ and $ V $ such that  $ \left [\begin{array}{cc}
UA_{11} & UA_{12}\\ VA_{21} & VA_{22}
\end{array}\right] $ is positive semidefinite.
\end{theorem}

\begin{theorem}\cite{Cheng-Horn-Li}\label{lm2.15}
Let $ C , C_1$ and $C_2$ be three square complex matrices of order $ n $ such that \break $ C=C_1+C_2 $. If $ S_j(\cdot{}) $ is the $j$-th singular value of corresponding matrix, then \break $ \sum\limits_{p}S_{p}(C)\leq \sum\limits_{p}S_{p}(C_1)+\sum\limits_{p}S_{p}(C_2)$.
\end{theorem}

%==================================================================

\section{Energy of a vertex of $ \mathbb{T} $-gain graphs}\label{ver-ener}

The energy of a vertex in an undirected graph is studied  in \cite{Vertexenergy1}. In this section, first we extend this notion for the $\mathbb{T}$-gain graphs, and establish some of the properties.
\begin{definition}\label{def1}
The energy of a vertex $ v_i $ of a $ \mathbb{T} $-gain graph $ \Phi $ is denoted by $ \mathcal{E}_{\Phi}(v_i) $ and is defined by
\begin{equation*}
\mathcal{E}_{\Phi}(v_i)=|A(\Phi)|_{ii}, \text{ for } i=1,2 \dots,n,
\end{equation*}
where $ |A(\Phi)|_{ii} $ is the $ (i,i) $-th entry of $ (A(\Phi)A(\Phi)^{*})^{\frac{1}{2}} $.
\end{definition}

It is easy to see that, the energy of a $ \mathbb{T} $-gain graph can be expressed as the sum of the energies of  vertices of $ \Phi $. That is,
\begin{equation}\label{eq2}
\mathcal{E}(\Phi)=\mathcal{E}_{\Phi}(v_1)+\mathcal{E}_{\Phi}(v_2)+\dots+\mathcal{E}_{\Phi}(v_n).
\end{equation}
	  	
Energy of a vertex of a $ \mathbb{T} $-gain graph can be obtained from the eigenvalues and the eigenvectors of $\Phi$. This is done in the next Lemma, and this result is an extension of Lemma \ref{lm-2.2} for the $\mathbb{T}$-gain graphs.  	
\begin{lemma} \label{lemma.1}
Let $ \Phi $ be a $ \mathbb{T} $-gain graph with the vertex set  $ \{v_1, v_2, \dots, v_n\}$. Then
\begin{equation*}
\mathcal{E}_{\Phi}(v_i)=\sum\limits_{j=1}^{n}Q_{ij}|\lambda_j|, \text{  for } i=1,2, \dots n.
\end{equation*}
where $ Q_{ij}=|q_{ij}|^{2} $ and $ Q=(q_{ij}) $ is the unitary matrix whose columns are the eigenvectors of $ \Phi $ and $ \lambda_j $ is the $ j $-th eigenvalue of $ \Phi $.
\end{lemma}
\begin{proof}
 Since $ A(\Phi) $ is Hermitian, so there exists a unitary matrix $ Q=(q_{ij}) $ such that $ A(\Phi) = Q D Q^{*}$, where $ D=\diag(\lambda_{1}, \lambda_{2}, \dots, \lambda_{n}) $. Therefore, the columns of $ Q $ are eigenvectors of $ A(\Phi) $. Now, it is easy to see that \begin{equation*}
\mathcal{E}_{\Phi}(v_i)=\sum\limits_{j=1}^{n}Q_{ij}|\lambda_j|, \text{  for } i=1,2, \dots n.
\end{equation*}
where $ Q_{ij}=|q_{ij}|^{2} $.
\end{proof}

Let $ \mathbb{C}_{n\times n} $ denote the set of all $ n\times n $ complex matrices. Consider the function $ \Omega_{i}: \mathbb{C}_{n\times n}\rightarrow \mathbb{C} $ such that $ \Omega_{i}(B)=b_{i,i} $, for $ i=1,2,\dots, n$,  where $ b_{i,i} $ is the $ (i,i)$-th entry of $ B$. Let $ \Phi=(G, \varphi) $ be a $ \mathbb{T} $-gain graph on $ n $ vertices with adjacency matrix $ A(\Phi) $. Then, $ \Omega_{i}(|A(\Phi)|)=\mathcal{E}_{\Phi}(v_i) $, $ i=1,2, \dots, n $. Now, it is clear that, for any two complex matrices $ B $  and $ C $, $ |\Omega_{i}(BC)|\leq \Omega_i(|BC|) $. Since $ \Omega_i $ is a positive linear functional, so the H$ \ddot{\text{o}}$lder inequality holds, see \cite{Vertexenergy1}. That is, if $ 0< s,t \leq \infty $ with $ 1=\frac{1}{s}+\frac{1}{t} $, then
\begin{equation} \label{eq8}
\Omega_{i}\big(|BC|\big)\leq \Omega_{i}\big(|B|^{s}\big)^{\frac{1}{s}}\Omega_{i}\big(|C|^{t}\big)^{\frac{1}{t}}.
\end{equation}
	 \begin{lemma}
	Let $ \Phi=(G, \varphi) $ be a $ \mathbb{T} $-gain graph on $ G $ of $ n $ vertices and at least one edge. If $ r\geq 2 $, $ 0< s, t<\infty $ such that $ \frac{1}{s}+\frac{1}{t}=1 $, then
	\begin{equation}\label{eq4}
	\frac{\Big(\Omega_{p}\big(|A(\Phi)|^{r}\big)\Big)^{t}}{\Big(\Omega_{p}\big(|A(\Phi)|^{s(r-1)+1}\big)\Big)^{\frac{t}{s}}} \leq \mathcal{E}_{\Phi}(v_p), \qquad p=1,2, \dots, n.
	\end{equation}
	 \end{lemma}
	 \begin{proof}
	 Let $ B=|A(\Phi)|^{r-\frac{1}{t}} $ and $ C=|A(\Phi)|^{\frac{1}{t}} $. Then, by the H$ \ddot{\text{o}}$lder inequality \eqref{eq8}, we have
	
	 \begin{equation*}
	 \Omega_{p}\big(|A(\Phi)|^{r}\big)=\Omega_{p}\big(|A(\Phi)|^{r-\frac{1}{t}}|A(\Phi)|^{-\frac{1}{t}}\big)\leq\Big(\Omega_{p}\big(|A(\Phi)|^{sr-\frac{s}{t}}\big)^{\frac{1}{s}}\Big)\Omega_{p}\Big(|A(\Phi)|\Big)^{\frac{1}{t}}
	 \end{equation*}
	 That is,
	 \begin{equation*}
	 \Big(\Omega_{p}\big(|A(\Phi)|^{r}\big)\Big)^{t}
	 \leq\Big(\Omega_{p}\big(|A(\Phi)|^{sr-\frac{s}{t}}\big)^{\frac{t}{s}}\Big)\Omega_{p}\big(|A(\Phi)|\big)
	 \end{equation*}
	 Since $ 0< s,t<\infty $, and $ \frac{1}{s}+\frac{1}{t}=1 $, so $ rs-\frac{s}{t}=s(r-1)+1 $. Therefore,
	 \begin{equation*}
	 \frac{\Big(\Omega_{p}\big(|A(\Phi)|^{r}\big)\Big)^{t}}{\Big(\Omega_{p}\big(|A(\Phi)|^{s(r-1)+1}\big)\Big)^{\frac{t}{s}}} \leq \Omega_{p}(|A(\Phi)|)=\mathcal{E}_{\Phi}(v_p), \qquad p=1,2 \cdots, n.
	 \end{equation*}
	 \end{proof}
	
	 Let $ M_{k}(\Phi,p) $ denote the sum of the gains of directed $ k $-walk from the vertex $ v_p $ to itself,  in the $ \mathbb{T} $-gain graph $ \Phi $. In the next result, we establish a  bound of the vertex energy  $ \mathcal{E}_{\Phi}(v_p) $ in terms of $ M_{k}(\Phi, p) $ and the vertex degree .

\begin{lemma} \label{Lm6}
Let $ \Phi=(G, \varphi) $ be any $ \mathbb{T} $-gain graph of $ n $ vertices  with at least one edge. Then
\begin{equation*}
\frac{d_p^{\frac{3}{2}}}{M_{4}(\Phi, p)^{\frac{1}{2}}} \leq  \mathcal{E}_{\Phi}(v_p), \text{ for all $ p=1,2, \dots, n. $}
\end{equation*}
\end{lemma}
\begin{proof}

In the Inequality \eqref{eq4}, we substitute $ r=2,s=3 $ and $ t=\frac{3}{2} $. Then we get
\begin{equation*}
\frac{\Omega_{p}(A(\Phi)^{2})^{\frac{3}{2}}}{\Omega_{p}(A(\Phi)^{4})^{\frac{1}{2}}} \leq \mathcal{E}_{\Phi}(v_p), \qquad p=1,2, \dots, n.
\end{equation*}
Since $\Omega_{p}(A(\Phi)^{4})=M_{4}(\Phi, p)  $ and $ \Omega_{p}(A(\Phi)^{2})=d_p $, So the corollary follows.
\end{proof}

Now, we establish a  bound for $ \mathcal{E}_{\Phi}(v_j)$ for $ \mathbb{T} $-gain graph in terms of vertex degree of $ v_j $ and the largest vertex degree $ \Delta(G) $. For undirected graph $ G $, these results are presented in \cite{Vertexenergy1}.

\begin{theorem}\label{th7}
Let $ \Phi=(G, \varphi) $ be any connected $ \mathbb{T} $-gain graph with at least one edge. Then
\begin{equation}\label{ver-enr-deg-bou}
\mathcal{E}_{\Phi}(v_j) \geq \sqrt{\frac{d_j}{\Delta(G)}}, \qquad\text{for all $ v_j \in V(G) $.}
\end{equation}
 Equality holds if and only if $ \Phi\sim(K_{d_j, \Delta(G)}, 1) $.
\end{theorem}
\begin{proof}
Let $ V(G)=\{v_1, v_2, \dots, v_n \} $ be the vertex set of $ G $. Let the degree of the vertex $ v_j $ be $ d_j $. Set $ d_j=d $. Then the following three types of directed $ 4 $-walks, starting from the vertex $ v_j $ to itself, are possible:

\begin{enumerate}
 \item $ v_j \rightarrow v_i\rightarrow v_j \rightarrow v_k \rightarrow v_j $;
 \item  $ v_j \rightarrow v_i\rightarrow v_s \rightarrow v_i \rightarrow v_j $, where $v_j \ne v_s $;
 \item $ v_j \rightarrow v_i\rightarrow v_s \rightarrow v_k \rightarrow v_j $, where four vertices are mutually distinct.
\end{enumerate}

Now the maximum value of the sum of the gains of the walks of type 1 is $ d^{2} $. Similarly, for the type 2, the maximum value is $ d(\Delta(G)-1) $, and for the type 3, the maximum value is $ 2 \sum\limits_{t=1}^{p} \cos (\theta_t)$, where $ p \leq \frac{d(\Delta(G) -1)(d-1)}{2} $ and $ \varphi(\overrightarrow{C_m})=e^{i\theta_m} $, $ \overrightarrow{C_m} $ is a $ 4 $-cycle formed by this walk. Thus the maximum value is $ d(\Delta(G) -1)(d-1) $, and hence $ M_{4}(\Phi, j) \leq d^{2}\Delta(G)$.  Now,  by Lemma \ref{Lm6}, we have $ \mathcal{E}_{\Phi}(v_j) \geq \sqrt{\frac{d_j}{\Delta(G)}} $.
	  	
If equality occurs in (\ref{ver-enr-deg-bou}), then $ M_{4}(\Phi, j)= d^{2}\Delta(G)$.  Therefore, $ G=K_{d, \Delta(G)} $. Again from the equality $ M_{4}(\Phi, j)= d^{2}\Delta(G)$, we have $ \varphi(\overrightarrow{C_m}) =1$, for all cycle passing through the vertex $ v_j $. Thus, by the Lemma \ref{lm3}, $ \Phi $ is balanced. Hence $ \Phi \sim (K_{d, \Delta(G)}, 1) $. Converse is easy to verify.

\end{proof}

\begin{corollary} \label{Cor.1}
Let $ \Phi=( G, \varphi) $ be any connected $ \mathbb{T} $-gain graph on a $ r $-regular graph $ G $. Then,
\begin{equation*}
\mathcal{E}_{\Phi}(v_i) \geq 1,
\text{ for all $ v_i \in  V(G) $}.
\end{equation*}
Equality occurs if and only if $\Phi=(K_{r,r}, 1) $.
\end{corollary}
	  	
\begin{proof}
Since $ G $ is a connected $ r $-regular graph with vertex set $ V(G)=\{ v_1, v_2, \dots, v_n \}$, so the degree of each vertex is same. For $ i=1, 2, \cdots, n $, $ d_i=r=\Delta(G) $. Then by the Theorem \ref{th7}, we  have $ \mathcal{E}_{\Phi}(v_i) \geq 1,
  $ for all $ v_i \in  V(G)$. Equality occur if and only if $ \Phi \sim (K_{r,r}, 1) $
\end{proof}
	  	In the next lemma, we show that the energy of a vertex is invariant under the switching equivalence of $\mathbb{T}$-gain graphs.
\begin{lemma}\label{Lm.4}
Let $ \Phi_{1} $ and $ \Phi_2 $ be any two switching equivalent  $ \mathbb{T} $-gain graphs on a  graph $ G $ with the vertex set $V(G)= \{v_1, v_2, \dots, v_n  \} $. Then  for each $ i $,
\begin{equation*}
\mathcal{E}_{\Phi_{1}}(v_i)=\mathcal{E}_{\Phi_{2}}(v_i).
\end{equation*}
\end{lemma}
\begin{proof}
Since $ \Phi_{1} \sim  \Phi_{2}$, so  $ \spec(\Phi_{1})=\spec(\Phi_{2}) $ and there is a diagonal unitary matrix $ U $ such that $ A(\Phi_{1})=UA(\Phi_{2})U^{*} $. Hence, by the Lemma \ref{lemma.1}, we have $ \mathcal{E}_{\Phi_{1}}(v_i)=\mathcal{E}_{\Phi_{2}}(v_i) $ for each $ i $.
\end{proof}
	
	  	In the next lemma, we provide a sufficient condition for the vertex energy of a  $\mathbb{T}$-gain graph equals to the vertex energy of its underlying graph.
\begin{lemma} \label{cor1}
Let $ \Phi=(G,\varphi) $ be any $ \mathbb{T} $-gain graph such that either $ \Phi $ is balanced or $ -\Phi $ is balanced. Then $ \mathcal{E}_{\Phi}(v_i) =\mathcal{E}_{-\Phi}(v_i)=\mathcal{E}_{G}(v_i)$.
\end{lemma}	
\begin{proof}
If $ \Phi $ is balanced, then $ \Phi \sim G $. Thus by the Lemma \ref{Lm.4}, $ \mathcal{E}_{\Phi}(v_i)=\mathcal{E}_{G}(v_i) $, for $ i=1,2, \dots, n $. Let $ \{\lambda_1, \lambda_2, \cdots, \lambda_n\} $ be the spectrum of $ \Phi $. Let $ D=diag(\lambda_1, \lambda_2, \cdots, \lambda_n) $. Then there exist an unitary matrix $ Q $ such that $ A(\Phi)=QDQ^{*} $. Thus $ A(-\Phi)=-A(\Phi)=Q(-D)Q^{*} $. Therefore, by the Lemma \ref{lemma.1}, $ \mathcal{E}_{\Phi}(v_i)=\mathcal{E}_{-\Phi}(v_i) $, for $ i=1, 2, \dots, n$. Thus $ \mathcal{E}_{\Phi}(v_i) =\mathcal{E}_{-\Phi}(v_i)=\mathcal{E}_{G}(v_i) $. If $ -\Phi $ is balanced then we can prove the statement similarly.
\end{proof}

	In the next Theorem, we provide a lower bound for the vertex energy  of a $\mathbb{T}$-gain graph in terms of the degree of the vertex  and the maximum vertex degree of the underlying graph. 	  	 
\begin{theorem}\label{Th.5}
Let $ \Phi=(G, \varphi) $ be any connected  $ \mathbb{T} $-gain graph with  at least one edge. Then
\begin{equation*}
\mathcal{E}_{\Phi}(v_i) \geq \frac{d_i}{\Delta(G)}, ~~~~\text{ for all  $ v_i \in V(G).$}
\end{equation*}
Equality occurs if and only if  $ \Phi \sim (K_{d,d},1) $, for some $d$.
\end{theorem}
\begin{proof}
Let $ \Phi=(G, \varphi) $ be any $ \mathbb{T} $-gain graph on $ G $. Let $ \lambda_{n}\leq \lambda_{n-1}\leq \dots \leq \lambda_{1} $ be the eigenvalues of $ \Phi $. By Theorem \ref{th-2.1}, $ \max\{\lambda_1,-\lambda_n\} =\rho(\Phi)\leq \rho(G)\leq \Delta(G).$ Hence $ \lambda_{i}\in [-\Delta(G), \Delta(G)] $ for all $ i $. Therefore, $ \big| \frac{\lambda_{i}}{\Delta(G)}\big| \leq 1 $. Then $\big| \frac{\lambda_{i}}{\Delta(G)}\big| \geq \big(\frac{\lambda_{i}}{\Delta(G)} \big)^{2}  $ and equality occur if and only if $ \lambda_{i} \in \{ -\Delta(G), 0, \Delta(G)\} $.
Using Lemma \ref{lemma.1}, we have
\begin{equation*}
\mathcal{E}_{\Phi}(v_i) =\sum\limits_{j=1}^{n}|q_{ij}|^{2}|\lambda_{j}| \geq \sum\limits_{j=1}^{n}|q_{ij}|^{2}\frac{\lambda_{j}^{2}}{\Delta(G)}=\frac{d_i}{\Delta(G)},  \text{ for all $ v_i\in V(G) $},
\end{equation*}

where $ Q=(q_{ij}) $ is the unitary matrix whose columns are eigenvectors of $ \Phi $.  Since $ G $ has at least one edge so there is a vertex $ v_j $ such that $ \mathcal{E}_{\Phi}(v_j)>0 $. Therefore, if equality occur then either $ \Delta(G)  $  or $ -\Delta(G) $ must be an eigenvalue of $ \Phi $.

Now $ \Delta(G)=\rho(\Phi)\leq \rho(G)\leq \Delta(G) $, so $ \rho(G)=\rho(\Phi) $. Thus, by Theorem \ref{th-2.1}, either  $ \Phi $ is balanced or $ -\Phi $ is balanced. \\
\noindent \textbf{Case-I:} If $ \Phi $ is balanced, then by Lemma \ref{Lm.4},  $ \mathcal{E}_{\Phi}(v_i)= \mathcal{E}_{G}(v_i)$ for all $ v_i \in V(G) $.
Then $ \mathcal{E}_{G}(v_i)=\frac{d_i}{\Delta(G)} $. Therefore, by Lemma \ref{th3}, $ G $ is isomorphic to $ K_{d,d} $, for some $ d $. Hence $ \Phi \sim (K_{d,d}, 1) $.
	  	
\noindent\textbf{Case-II:} If $ -\Phi $ is balanced,
then, similar to case-I, $ -\Phi \sim (K_{d,d},1) $. Since the underlying graph is bipartite and $ -\Phi $ is balanced, so, by Theorem \ref{th-2.2}, $ \Phi $ is balanced. Thus $ \Phi \sim (K_{d,d}, 1) $.
\end{proof}
Using the above theorem, we prove that the energy of the complete bipartite $\mathbb{T}$-gain graph  $K_{n,n}$ is always greater than or equal to the energy of the underlying graph.
	\begin{theorem} \label{th5}
	  	If $ \Phi=(K_{n,n}, \varphi) $ is any $ \mathbb{T} $-gain graph on the complete bipartite graph $ K_{n,n} $, then  $ \mathcal{E}(\Phi) \geq \mathcal{E}(K_{n,n})=2n$, and equality holds if and only if $ \Phi \sim (K_{n,n},1) $.
	 	\end{theorem}
	 	\begin{proof}
	  	Let  $V(G)= \{v_1, v_2, \dots, v_{2n}  \} $ be the set of vertices of $ \Phi $. Then,
	  	$\mathcal{E}(\Phi)=\mathcal{E}_{\Phi}(v_1)+\mathcal{E}_{\Phi}(v_2)+\dots+\mathcal{E}_{\Phi}(v_{2n}). $ By  Theorem \ref{Th.5}, we have
	  	\begin{align*}
	  	\mathcal{E}(\Phi)&=\sum\limits_{j=1}^{2n}\mathcal{E}_{\Phi}(v_j)\geq\sum\limits_{j=1}^{2n}\frac{d_j}{\Delta(G)}=2n=\mathcal{E}(K_{n,n}).\\
	  	\end{align*}
	  	It is easy to see that, equality occur if and only if $ \Phi $ is balanced.
	    \end{proof}

If $ G $ is a $ r $-regular graph of $ n $ vertices, then $ \mathcal{E}(G) \geq n $ and equality occur if and only if each component of $ G $ is $ K_{r,r} $ [Theorem \ref{lm.19}]. Next corollary is an extension of the above result for the $ \mathbb{T} $-gain graph.

\begin{corollary}
Let $ \Phi=(G, \varphi) $ be any $ r $-regular $ \mathbb{T} $-gain graph on $ n $ vertices, where $ r>0 $. Then $ \mathcal{E}(\Phi) \geq n $ and equality occur if and only if each component of $ \Phi $ is switching equivalent to $ (K_{r,r}, 1) $.
\end{corollary}
\begin{proof}
Let $ G_1, G_2, \dots , G_k $ be the connected components of $ G $. Then $ \mathcal{E}(\Phi)=\sum\limits_{j=1}^{k}\mathcal{E}((G_j, \varphi)) $. By the definition \ref{def1} and the equation (\ref{eq2}),   $ \mathcal{E}((G_j, \varphi))=\sum\limits_{v\in V(G_j)}\mathcal{E}_{(G_j, \varphi)}(v) $, for each $ j=1, 2, \dots , k $.
 Now, by Corollary \ref{Cor.1}, $\mathcal{E}_{(G_j, \varphi)}(v)\geq 1  $, for any $ v\in V(G_j) $. Then $ \mathcal{E}(G_j, \varphi) \geq |V(G_j)| $, for all $ j=1, 2, \dots, k $. Thus $ \mathcal{E}(\Phi)\geq n $.
 If $ \mathcal{E}(\Phi)=n $, then $ \mathcal{E}_{\Phi}(u)=1$, for all $ u\in V(G) $. Therefore, by Corollary \ref{Cor.1}, each component of $ \Phi $ is switching equivalent to $ (K_{r,r},1) $.
\end{proof}

Let  $ G $ be a semiregular bipartite graph with parameter $ ( n_a,n_b,r_a,r_b) $. A  \emph{semi regular bipartite $ \mathbb{T} $-gain graph} with parameter $ ( n_a,n_b,r_a,r_b) $ is a $ \mathbb{T} $-gain graph whose underlying graph is a semiregular bipartite graph of parameter $ ( n_a,n_b,r_a,r_b) $. The next bound is the generalization of a Theorem \ref{lm.20} for the $ \mathbb{T} $-gain graphs.
	
\begin{corollary}
Let $ \Phi=(G, \varphi) $ be any semiregular bipartite  $ \mathbb{T} $-gain graph with parameter $ (n_a, n_b, r_a, r_b) $. Then $ \mathcal{E}(\Phi) \geq n_a \sqrt{\frac{r_a}{r_b}}+n_b\sqrt{\frac{r_b}{r_a}} $. Equality  occur if and only if each component is switching equivalent to $ (K_{r_a, r_b}, 1) $.
\end{corollary}
\begin{proof}
Let $ G_1, G_2, \dots, G_k $ be the connected components of $ G $. Then each $ G_j $ is connected semiregular bipartite graph. Now, by applying Theorem \ref{th7} to each component $ (G_j, \varphi) $, we get the result.
\end{proof}

\begin{remark}
Hermitian adjacency matrices of  mixed graphs are particular case of adjacency matrices of $ \mathbb{T} $-gain graphs. Therefore all the above results for energy of a vertex holds true for  mixed graphs.
\end{remark}

% =======================================================================================	
	  	
\section{ Lower bounds of energy of $ \mathbb{T} $-gain graphs}\label{lower-bound}

	In this section, we establish several lower bounds for the energy of $\mathbb{T}$-gain graphs. We begin this section with the following theorem which gives a lower bound for the energy of a  $\mathbb{T}$-gain graph in terms of the gain of the real parts of the fundamental cycles.

\begin{theorem}\label{th-4.1}
Let $ \Phi=(G,\varphi)$ be any connected $ \mathbb{T} $-gain graph on $ n $ vertices. Let $T$ be a normal spanning tree of $G$, and $ \{ C_1, C_2, \dots, C_l\} $ be the collection of all fundamental cycles in $ G $ with respect to  $T$. Then,
\begin{equation}\label{ene-bou-fun-cyc}
\mathcal{E}(\Phi) \geq 2\sum\limits_{j = 1}^{l} \Re \big(\varphi (C_j)) \big) +(5n-n^2-4).
\end{equation}
The inequality is sharp.
\end{theorem}
\begin{proof}
Let $ \Phi=(G, \varphi) $ be any connected $ \mathbb{T} $-gain graph. Let $T$ be a normal spanning tree of $G$, and $ \{ C_1, C_2, \dots, C_l\} $ be the collection of all fundamental cycles in $ G $ with respect to  $T$.  Define a new $ \mathbb{T} $-gain graph $ \Phi^{'} $ on $G$ such that  $ \varphi^{'}(\overrightarrow{e})=1 $ for all $ e\in E(T) $ and $\varphi(C_i) = \varphi^{'}(C_i)$ for all $i$.
So, by Lemma \ref{lm-2.1},  the $ \mathbb{T} $-gain graphs $ \Phi$ and $ \Phi^{'}$ are switching equivalent.  Therefore,  \begin{equation}\label{ga-re}
\sum\limits_{i,j} \varphi^{'}(\overrightarrow{e_{i,j}})=\sum\limits_{j=1}^{l}\{\varphi^{'}(\overrightarrow{C}_{j})+\varphi^{'}(\overrightarrow{C}_{j})^{-1}\}+2(n-1)=2\sum\limits_{j=1}^{l}\Re\big(\varphi(C_{j}) \big)+2(n-1) .
\end{equation}
Now,  $ \Re\big( \tr(A(K_n)A(\Phi^{'})) \big) = \sum\limits_{ i,j} \varphi^{'}(\overrightarrow{e_{i,j}})$. Let $ |\lambda_{1}| \geq |\lambda_{2}| \geq \dots \geq| \lambda_{n}| $ be the singular values of $ \Phi^{'} $. Since $ \spec(K_n)=\{-1^{(n-1)}, (n-1)^{(1)} \} $, by  Theorem \ref{th1}, we have  $\Re(\tr(A(K_n)A(\Phi^{'})) \leq
 (n-1)|\lambda_{1}|+\sum\limits_{j=2}^{n-1}|\lambda_j|$, and hence  $\sum\limits_{ i,j} \varphi^{'}(\overrightarrow{e_{i,j}})
 \leq(n-2)|\lambda_{1}|+\mathcal{E}(\Phi^{'})$. So, by equation (\ref{ga-re}), we have
$2\sum\limits_{j=1}^{l}\Re\big(\varphi(C_{j}) \big)+2(n-1)\leq(n-2)|\lambda_{1}|+\mathcal{E}(\Phi^{'}).$
Now,  using Theorem \ref{th-2.1},  we get  $|\lambda_{1}|=\rho(\Phi^{'})=\rho(\Phi)\leq \rho(G)\leq \Delta\leq  (n-1)  $.  Thus
$$\sum\limits_{j=1}^{l}\Re\big(\varphi(C_{j})\big)+2(n-1)\leq(n-2)(n-1)+\mathcal{E}(\Phi),$$
 and hence $$\mathcal{E}(\Phi) \geq 2\sum\limits_{j=1}^{l}\Re\big(\varphi(C_{j}) \big) +(5n-n^2-4).$$
Now, if $ \Phi\sim(K_n,1) $, then $2\sum\limits_{j=1}^{l}\Re\big(\varphi(C_{j}) \big)= n^2-3n+2 $ and hence equality holds in equation (\ref{ene-bou-fun-cyc}).
\end{proof}

If $G$ is a complete bipartite graph, then $A(G)$ has exactly one positive eigenvalue. Also, if $G$ is any non-complete bipartite graph on more than $4$ vertices, then it contains $P_4$ as an induced subgraph, and hence $A(G)$ has at least two positive eigenvalues. So, if  $G$  is a bipartite graph on more than $4$ vertices, then $G$ is complete bipartite if and only if $A(G)$ has exactly one positive eigenvalue.  Our next objective is to study the counter part of this property for the  $\mathbb{T}$-gain graphs. The following lemma is a key in the proof of Theorem \ref{lem5}. This gives a sufficient condition a $\mathbb{T}$-gain graph to be balanced.
\begin{lemma} \label{lm3}
Let $ \Phi=(G, \varphi) $ be any $ \mathbb{T} $-gain graph on a complete bipartite graph $ G $. If  every $ 4 $-cycle which  passes through the vertex $v$, for some vertex $ v $ of $ G $,  has gain $1$, then $ \Phi $ is balanced.
\end{lemma}
\begin{proof}
Let $ \Phi=(G, \varphi) $ be any $ \mathbb{T} $-gain graph on a complete bipartite graph $ G $. Let $ v $ be a vertex in $ G $ such that the gain of any $ 4 $-cycle passing through the vertex $ v $ is $ 1 $. First let us show that gain of any four cycle in $ G $ is $1$. Let $ C_4 \equiv v_1-v_2-v_3-v_4-v_1$ be any $ 4 $-cycle in $ \Phi$ such that $C_4$ does not contain the vertex $v$. Without loss of generality, let us assume that $v_2 \sim v$. Now, consider the two $ 4 $-cycles: $ C_4(v)\equiv v_1-v_2-v-v_4-v_1 $ and $C_4^{'}(v)\equiv v-v_2-v_3-v_4-v.$ Then $ \varphi(\overrightarrow{C_4})=\varphi(\overrightarrow{C_4(v)})\varphi(\overrightarrow{C_4^{'}(v)})=1 $.
	  	  	
Let $C_{2p}$ be any cycle in $G$ on $ 2p $ vertices. Without loss of generality, let us assume that $C_{2p}\equiv v_1-v_2-v_3-v_4-\dots-v_{(2p-1)}-v_{2p}-v_1$. Then
\begin{align*}
\varphi(\overrightarrow{C_{2p}})=& \varphi(\overrightarrow{e_{1,2}})\varphi(\overrightarrow{e_{2,3}})\varphi(\overrightarrow{e_{3,4}})\dots \varphi(\overrightarrow{e_{(2p-1),2p}})\varphi(\overrightarrow{e_{2p,1}})\\
 =& \left\{ \varphi(\overrightarrow{e_{1,2}})\varphi(\overrightarrow{e_{2,3}})\varphi(\overrightarrow{e_{3,4}})\varphi(\overrightarrow{e_{4,1}}) \right\} \\
&\left\{ \varphi(\overrightarrow{e_{1,4}})\varphi(\overrightarrow{e_{4,5}})\varphi(\overrightarrow{e_{5,6}})\varphi(\overrightarrow{e_{6,1}})\right\} \\
& \quad \quad \quad \vdots\\ &\left\{\varphi(\overrightarrow{e_{1,(2p-2)}})\varphi(\overrightarrow{e_{(2p-2),(2p-1)}})\varphi(\overrightarrow{e_{(2p-1),2p}})\varphi(\overrightarrow{e_{2p,1}})\right\}\\
=&1.
\end{align*}
Thus $\Phi$ is balanced.
\end{proof}

\begin{theorem} \label{lem5}
Let $ \Phi=(G,\varphi) $ be any $ \mathbb{T} $-gain graph on a connected bipartite graph $ G $. Then $ \Phi $ has exactly one positive eigenvalue if and only if $ \Phi $ is a balanced complete bipartite graph.
\end{theorem}
\begin{proof}
Let $ \Phi = (G,\varphi) $ have exactly one positive eigenvalue.
If the number of vertices of $ G $ is two or three, then $ G $ must be $ K_2 $ or $ K_{1,2} $, respectively. Therefore, in both the cases, $ \Phi $ is a balanced compete bipartite graph. Now we consider a graph $ G $ with $ |V(G)|\geq 4 $. Suppose that $ P_4 $ is an induced subgraph of $ G $. So $ (P_4, \varphi) $ is an induced $ \mathbb{T} $-gain subgraph of $ \Phi $. As $P_4$ is a tree, so the spectrum of $P_4$ with respect to $\varphi$ is same as that of $\spec(P_4)$. Thus $ P_4 $ has two positive eigenvalue with respect to $\varphi$. Therefore, by the interlacing theorem, $ \Phi $ has at least two  positive eigenvalue,  a contradiction. Thus $ G $ can not have $ P_4 $ as an induced subgraph, and hence the diameter of $ G $ is at most $ 2 $. Now it is easy to see that any two non adjacent vertices have the same neighbors. Thus $ G $ is complete multipartite. But $ G $ is bipartite, so $ G $ is complete bipartite.

 Now we consider the following two cases to show that $\Phi$ is balanced.\\
{\bf Case 1:} If $ G $  does not contain any cycles, then  $ G $ must be a star, and hence $ \Phi $ is  balanced.\\
{\bf Case 2:} If $ G $ contains cycles, then it must contains an induced $ C_4 $. Let $ \varphi(\overrightarrow{C_4})=e^{i\theta}, \theta \in [0, 2\pi)$.
	  	
Let $ C =( C_4, \varphi) $ be an induced subgraph of $ \Phi $ whose underlying graph is $ C_4 $.
Therefore, by  Theorem \ref{th-2.3}, we have
\begin{align*} \spec(C)=\left\{2\cos\left(\frac{\theta}{4}\right),2\cos\left(\frac{\theta}{4}+\frac{\pi}{2}\right),2\cos\left(\frac{\theta}{4}+\pi\right),2\cos\left(\frac{\theta}{4}+\frac{3\pi}{2}\right)\right\}
\end{align*}
Let $x=\frac{\theta}{4}\in[0,\frac{\pi}{2})$. It is easy to see that $\spec(C)$ has two positive and two negative eigenvalues if and only if $x\in(0,\frac{\pi}{2})$. Hence $\spec(C)$ has exactly one positive eigenvalue if and only if $x=0$. Now,  by interlacing theorem, $\Phi$ cannot have any induced $ 4 $-cycle $ C_4 $ such that $\varphi(\overrightarrow{{C_4}})=e^{i\theta}$,  where  $\theta\in(0,2\pi)$.
Therefore, for any induced $ 4 $-cycle $ C $ in $ G $, we have $ \varphi(C)=1 $. Thus, by  Lemma \ref{lm3}, $ \Phi $ is  balanced.
Conversely, if $\Phi$ is a balanced complete bipartite $ \mathbb{T} $-gain graph, then $\Phi$ has exactly one positive eigenvalue.
\end{proof}
	  	
Next result gives a lower bound of  energy of $ \mathbb{T} $-gain graph in terms of spectral radius.	
	  	
\begin{theorem} \label{Th6}
If $ \Phi=(G, \varphi) $ be any $ \mathbb{T} $-gain graph on a  connected graph $ G $. Then $ \mathcal{E}(\Phi) \geq 2 \rho(\Phi) $. If $ G $ is bipartite then equality occurs if and only if $ \Phi\sim(K_{p,q},1)$ for some $ p, q $.
\end{theorem}
\begin{proof}
Let $ \{ \lambda_1, \lambda_2, \dots, \lambda_n\} $ be the spectrum of $ \Phi $ such that $ \lambda_1 \geq \lambda_2 \geq \dots \geq \lambda_n $. Now $\lambda_1 + \lambda_2 + \dots +\lambda_n =0$. Therefore, $ 2|\lambda_1| \leq |\lambda_1|+|\lambda_2|+\dots +|\lambda_n| $. Thus $ \mathcal{E}(\Phi) \geq 2\rho(\Phi) $.
%[\textbf{I think it can happen that $\rho(\Phi) = \vert \lambda_n\vert$. If this is the case, then we may conclude $-\Phi\sim(K_{p,q},1)$. Please check this. }]
	  	
Let $ G $ be bipartite. Since $ |\lambda_1|=|\lambda_2|+\dots +|\lambda_n| $ holds if and only if all of $ \lambda_j's $, for $ j=2, 3, \dots, n $ are of the same sign. Therefore, equality occur if and only if $ \Phi $ has only one positive eigenvalue. Hence, by the Theorem \ref{lem5}, equality holds if and only if $ \Phi\sim(K_{p,q}, 1) $.
\end{proof}

Let $ J $ denote the all $1$'s matrix of appropriate size. The following two theorems provide a lower bound for energy of $ \mathbb{T} $-gain graph in terms of the number of vertices and the gains of fundamental cycles.
	
\begin{theorem} \label{lm2}
 Let $ \Phi=(G, \varphi) $ be any connected  $ \mathbb{T} $-gain graph with $ n $ vertices and $ \{ C_1, C_2, \dots, C_l\} $ be the collection of all fundamental cycles in $ G $ with respect to a  normal spanning tree $ T $. Then
\begin{equation}\label{eq3}
\mathcal{E}(\Phi) \geq  4+ \frac{4}{n} \left\{ \sum_{j=1}^{l}\Re(\varphi(C_j))-1 \right\}.
\end{equation}
The inequality is sharp.
\end{theorem}  	
\begin{proof}
 Let $ \{ \lambda_1, \lambda_2, \dots, \lambda_n \} $ be the spectrum of $ \Phi $ such that $ |\lambda_1|\geq |\lambda_2|\geq \dots \geq |\lambda_n| $.  Define a  new  $ \mathbb{T} $-gain $ \Phi^{'} $ on $G$ such that  $ \varphi^{'}(\overrightarrow{e})=1 $ for all $ e\in E(T) $ and $\varphi(C_i) = \varphi^{'}(C_i)$ for all $i$.
So, by Lemma \ref{lm-2.1},  the $ \mathbb{T} $-gain graphs $ \Phi$ and $ \Phi^{'}$ are switching equivalent.  Then  $ \sum\limits_{ i,j} \varphi^{'}(\overrightarrow{e_{ij}})=2\sum\limits_{j=1}^{l}\Re\big(\varphi(C_{j}) \big)+2(n-1) $.   By Theorem \ref{th1}, we have $$\Re (\tr( A(\Phi) J)) \leq n |\lambda_1|,$$ and hence
$$2(n-1)+2\sum_{j=1}^{l}\Re(\varphi(C_j))  \leq n |\lambda_1|.$$
As $ |\lambda_{1}|\leq |\lambda_{2}|+\dots +|\lambda_n| $, so $ |\lambda_1|\leq \frac{\mathcal{E}(\Phi)}{2} $. Therefore,
\begin{equation*}
 \mathcal{E}(\Phi) \geq 4+ \frac{4}{n} \left\{ \sum_{j=1}^{l}\Re(\varphi(C_j))-1 \right\}.
\end{equation*}

Let us take $ \Phi \sim (G, 1) $, where $ G=K_{r,r,\cdots, r}$ is a connected complete $ p $-partite graph on $ m $ edges and $ n $ vertices. Then the right hand side expression (\ref{eq3}) becomes $ 4+\frac{4}{n}\left( m-n+1-1\right) $, which is $ \frac{4m}{n} $. Since $ G $ is a complete multipartite graph, $ G $ has exactly one positive eigenvalue. Thus, if $\lambda_1,\dots,\lambda_n$ are the eigenvalues of $\Phi$,  then $ |\lambda_1|=|\lambda_2|+\cdots +|\lambda_n| $, and hence $ \mathcal{E}(\Phi)=2|\lambda_1|=2\lambda_1$.  Also the spectral radius of $\Phi$ is $ (r-1) $, the degree of each vertex in $G$, and the  degree of each vertex of $ G $ is $ \frac{2m}{n} $. Therefore, $ \mathcal{E}(\Phi)=\frac{4m}{n} $. Hence the inequality is sharp.
\end{proof}
	  	
	  	If the underlying graph is a bipartite graph, then we can completely characterize the classes for which equality holds in (\ref{eq3}).
	  	
\begin{corollary}
Let $ \Phi=(G, \varphi)$ be any connected $ \mathbb{T} $-gain graph on a bipartite graph $ G $ with $ n $ vertices. Then
\begin{equation*}
\mathcal{E}(\Phi) \geq  4+ \frac{4}{n} \left\{ \sum_{j=1}^{l}\Re(\varphi(C_j))-1 \right\}.
\end{equation*}
Equality occurs if and only if $ \Phi \sim (K_{\frac{n}{2}, \frac{n}{2}}, 1) $.
\end{corollary}

\begin{proof}
Let $ \Phi=(G, \varphi) $ be any connected $ \mathbb{T} $-gain graph on a bipartite graph $ G $ with $ m $ edges and $ n $ vertices. Let $ \{\lambda_1, \lambda_2, \cdots, \lambda_n \}$ be the spectrum of $ \Phi $ such that $ |\lambda_1|\geq |\lambda_2|\cdots \geq |\lambda_n| $. Let $ \{ C_1, C_2, \cdots, C_l\} $ be the fundamental cycles of $ G $ with respect to a normal spanning tree $ T $. Now the
inequality is clear from the Theorem \ref{lm2}. Let us consider the equality,  $$\mathcal{E}(\Phi)=  4+ \frac{4}{n} \left\{ \sum_{j=1}^{l}\Re(\varphi(C_j))-1 \right\}.$$ Then from the proof of the Theorem \ref{lm2}, we have $$ \mathcal{E}(\Phi)\geq 2|\lambda_1|\geq  4+ \frac{4}{n} \left\{ \sum_{j=1}^{l}\Re(\varphi(C_j))-1 \right\} .$$  Since $\Phi $ is a bipartite $ \mathbb{T} $-gain graph , so $ \Phi $ must satisfy the following equality.
\begin{enumerate}
\item[(i)] $ \mathcal{E}(\Phi)=2|\lambda_1|=2\rho(\Phi)$
\item[(ii)] $ 2\lambda_1=2|\lambda_1|=4+ \frac{4}{n} \left\{ \sum_{j=1}^{l}\Re(\varphi(C_j))-1 \right\} $
\end{enumerate}
Now by the Theorem \ref{Th6}, the equation $ (i) $ is satisfied if and only if $ \Phi\sim (K_{p,q},1) $, for some $ p$ and $q $. Then $ \Phi $ is balanced, so from equation $ (ii) $, we have $ \lambda_1 =\frac{2m}{n}$. Since $ \Phi  \sim (K_{p,q}, 1)$, so $ \lambda_1=\sqrt{pq} $, $ n=p+q $ and $ m=pq $. Therefore, we have  $ \sqrt{pq}=\frac{2pq}{p+q} $. That is $ p=q=\frac{n}{2} $. Thus $ \Phi \sim (K_{\frac{n}{2}, \frac{n}{2}}, 1) $.  Converse is easy to verify.
\end{proof}

The following lemma is about the change in the energy of a graph obtained from a graph by removing a cut set. This will be useful in the proof of some of the following results.
\begin{lemma} \label{lm6}
Let $ \Phi=(G, \varphi) $ be a $ \mathbb{T} $-gain graph and $ E $ be a cut set of $ \Phi $. Then $ \mathcal{E}(\Phi-E) \leq \mathcal{E}(\Phi).$
\end{lemma}
\begin{proof}
For any cut set $ E $ of $ G $, there exist two induced sub graphs  $ L $ and $ M $ complement to each other in $ G $ such that  $ G-E=L \oplus M $.
Then $ \Phi-E=(L, \varphi)\oplus (M, \varphi) $. Now $ A(\Phi) $ can be expressed as $ \left[ \begin{array}{cc}
A((L, \varphi))& X \\ X^{*} & A((M, \varphi)) \end{array} \right] $. Therefore by the Theorem \ref{lm11}, $ \mathcal{E}(\Phi)\geq \mathcal{E}(A(L, \varphi))+\mathcal{E}(A(M, \varphi))=\mathcal{E}(\Phi-E) $.
\end{proof}

In the next result, we establish a connection between the gain energy and the matching number of a graph. This result is a counter part (for the $\mathbb{T}$-gain graphs) of  Lemma \ref{lm.14} and Theorem \ref{lm.16} for undirected graph, a main result in \cite{Tian-Wong} for skew energy of oriented graph, and Theorem \ref{lm.18} for mixed graph.

\begin{theorem}\label{th4}
Let $ \Phi=(G, \varphi) $ be a $ \mathbb{T} $-gain graph, and let $ \mu(G) $ be the  matching number of $G$. Then $ \mathcal{E}(\Phi) \geq 2\mu(G) $.
\end{theorem}	

\begin{proof}
Let $ \Phi=(G, \varphi) $ be any $ \mathbb{T} $-gain graph with matching number $ \mu(G) $. We prove the result by induction on $ \mu(G) $. If $ \mu(G)=0 $, then $ \mathcal{E}(\Phi)=2\mu(G)=0 $. If $ \mu(G)=1 $, then $ G $ must be  $ K_{1, p} $, for some $p$, together with some isolated vertices. Therefore, $ \Phi \sim (G, 1) $. Thus $ \mathcal{E}(\Phi)=2\sqrt{p}\geq 2=2\mu(G) $. Let us assume that for any $ \mathbb{T} $-gain graph $ \Psi=(H, \psi) $ with matching number $ \mu(H)<\mu(G) $, $ \mathcal{E}(\Psi)\geq 2\mu(H) $. Let $ M $ be a maximum matching of $ G $ and $ e\in M $. Now consider an induced subgraph $ G-[e] $. Then $ \mu(G-[e])= \mu(G)-1$. By induction, we have $ \mathcal{E}((G-[e], \varphi))\geq 2\mu(G-[e]) $. Let  $ E $ be the set of edges in $ G $  which are incident with the edge $e$. Then $E$ is a cut set, and $ (G-E)=(G-[e])\oplus K_2 $.  By the Lemma \ref{lm6}, $\mathcal{E}(\Phi)\geq  \mathcal{E}(\Phi-E)$. Now $ \mathcal{E}(\Phi)\geq \mathcal{E}(\Phi-E)=\mathcal{E}((G-[e], \varphi))+\mathcal{E}((K_2, \varphi)) \geq 2\mu(G)-2+2=2\mu(G)$. Hence the result.
\end{proof}

We shall discuss the sharpness of the inequality in the above bound  in Theorem \ref{th4(ii)}.

The following lemma is the counter part of Lemma \ref{lm.13} for the $\mathbb{T}$-gain  graphs.
\begin{lemma} \label{lm13}
Let $ \Phi=(G, \varphi) $ be a $ \mathbb{T} $-gain graph. If $ E $ is a cut set in $ G $ such that $V(G) = V_1\cup V_2$ and all the edges of $E$ are from the vertices of $V_1$ to a fixed vertex of $V_2$, then $ \mathcal{E}(\Phi-E)< \mathcal{E}(\Phi) $.
\end{lemma}
\begin{proof}
Let $ E $ be a cut set, and $ L $ and $ M $ be two complementary induced subgraphs in $ G $ corresponding to $E$. Let us assume that the edges of $ E $ are incidence with a single vertex  $ v $ of $ M $. After a suitable relabeling of vertices, we can express $ A(\Phi)=\left[ \begin{array}{cc}
A((L,\varphi)) & X\\ X^{*} & A((M, \varphi))\end{array} \right] $ such that the first column of the matrix $ X $, say $ y $, corresponds to the vertex $ v $. Hence all the entries of the matrix $ X $ are zero, except the first column. Now, by Lemma \ref{lm6}, $ \mathcal{E}(\Phi-E)\leq \mathcal{E}(\Phi) $. Suppose that $ \mathcal{E}(\Phi-E)=\mathcal{E}(\Phi) $. Then, by Theorem \ref{lm11}, there exists two unitary matrices $ P $ and $ Q $, such that $ \left[ \begin{array}{cc}
PA((L, \varphi)) & PX\\ QX^{*} & QA((M,\varphi))\end{array} \right]  $  is positive semi definite. As $(PX)^{*}=QX^{*}  $, we have $ Q=\left[ \begin{array}{cc}
\beta & 0\\0& Q_{1}\end{array} \right] $ with $ |\beta|=1 $, and $ Q_1 $ is unitary matrix. Let $ A((M, \varphi))=\left[ \begin{array}{cc}
0 & z^{*}\\ z & N\end{array} \right] $. Then $ QA((M,\varphi))=\left[ \begin{array}{cc}
0 & \beta z^{*}\\ Q_1z & Q_1N\end{array} \right] $ is positive semi definite. So  $ Q_1z=0 $ and  $ \beta z^{*}=0$. That is $ z=0 $. Hence $ A((M, \varphi))=\left[ \begin{array}{cc}
0 &0\\ 0 & N\end{array} \right] $. Therefore $ A(\Phi)=\left[ \begin{array}{ccc}
A((L,\varphi)) & y & 0\\y^{*}& 0 & 0\\ 0 & 0& N\end{array} \right] $. Now $ \mathcal{E}(A((L, \varphi)))+\mathcal{E}(N)=\mathcal{E}(A((L, \varphi)))+\mathcal{E}(A(M, \varphi)))=\mathcal{E}(\Phi-E) =\mathcal{E}(\Phi)=\mathcal{E}\left( \left[ \begin{array}{cc}
A((L,\varphi)) & y\\ y^{*} & 0\end{array} \right]\right)+\mathcal{E}(N)$. That is $ \mathcal{E}\left( \left[ \begin{array}{cc}
A((L,\varphi)) & y\\ y^{*}& 0\end{array} \right]\right) =\mathcal{E}(A((L, \varphi))) $. Hence, by Lemma \ref{lm12}, $ y=0 $. Thus $ E $ is empty. Which is a contradiction.

\end{proof}

The following lemma provides a (spectral) sufficient condition for a graph to have perfect matching.   This  is  a counter part of Lemma \ref{lm14} for the $
\mathbb{T}$-gain graphs.
\begin{lemma} \label{lm4}
 If $ \Phi=(G, \varphi) $ is a connected $ \mathbb{T} $-gain graph and  $ \mathcal{E}(\Phi)=2\mu(G) $, then $ G $ has a perfect matching
\end{lemma}
\begin{proof}
Suppose that  $ G $ has no perfect matching. Let $ M $ be any maximum matching of $ G $. Since $ G $ is a connected graph, so there exist a vertex $ u $ which is not adjacent with any edges in $ M $. Then $ \mu(G)=\mu (G-u) $. Let $ K_1 $ be the graph which is an isolated vertex $ u $. Let $ E $ be the set of all edges incident with the vertex $u$ in $G$. Then, $E$ is a cut set, and $ \Phi-E=(\Phi-u)\oplus K_1 $. Therefore, by  Lemma \ref{lm13} and  Theorem \ref{th4}, $ \mathcal{E}(\Phi) > \mathcal{E}(\Phi-E) =\mathcal{E}(\Phi-u)+0\geq 2\mu(G-u)=2\mu(G)$. That is, $ \mathcal{E}(\Phi) > 2\mu(G) $, a contradiction. Thus $ G $ has a perfect matching.
\end{proof}
	
Now, let us establish a couple of lemmas about the energy of a $\mathbb{T}$-gain graph in terms of the matching number of the underlying graph.
	
\begin{lemma} \label{lm5}
Let $ \Phi=(G, \varphi) $ be a connected $ \mathbb{T} $-gain graph with a pendant vertex. If $ G $ is not $ K_2 $, then $ \mathcal{E}(\Phi) > 2\mu(G)$.
\end{lemma}	
\begin{proof}
Let $ v $ be a pendent vertex of $ G $, $ u $ be its unique neighbor vertex, and $e$ be the edge between them.   Then the induced subgraphs $ (G-[e]) $ and $K_2  $ are complement to each other in $ G$.  Let $E$ be the collection of all edges between the vertex $u$ and the vertices of $G-\{u,v\}$ . Then  $ G-E=(G-[e] ) \oplus K_2 $.   By Lemma \ref{lm13}, $  \mathcal{E}(\Phi) > \mathcal{E}(\Phi-E)=\mathcal{E}(( G-[e], \varphi))+\mathcal{E}((K_2, \varphi))$. Also $ \mu(G-[e])=\mu(G)-1 $. Therefore, by Theorem \ref{th4}, $ \mathcal{E}(\Phi)> 2\mu(G).$
\end{proof}

\begin{lemma} \label{lm7}
Let $ \Phi=(G, \varphi) $ be a connected $ \mathbb{T} $-gain graph and $ L $ be an induced subgraph of $ G $. If $ \mathcal{E}((L, \varphi)) >2\mu(L) $ and $ \mu(G)=\mu(L)+\mu(G-L) $. Then $ \mathcal{E}(\Phi)> 2\mu(G) $.
\end{lemma}
\begin{proof}
Since $ L $ is an induced subgraph of $ G $, so $ (G-L) $ is the complementary induced subgraph of $ G $. Let $ E $ be a cut set of $ G $ such that $ (G-E)=(G-L)\oplus L $. Then $ (\Phi-E)=(G-L, \varphi)\oplus (L, \varphi) $. Since $ E $ is a cut set of $ G $, so, by  Lemma \ref{lm6}, $ \mathcal{E}(\Phi)\geq \mathcal{E}(\Phi-E)=\mathcal{E}((L, \varphi))+\mathcal{E}((G-L, \varphi)).$ Now, by  Theorem \ref{th4} and the hypothesis, we have $\mathcal{E}(\Phi) > 2\mu(L)+2\mu(G-L)=2\mu(G)$. Hence, $ \mathcal{E}(\Phi)> 2\mu(G).$
\end{proof}

\begin{lemma} \label{lm8}
Let $ \Phi=(G, \varphi) $ be any $ \mathbb{T} $-gain graph on a connected graph $ G $ which is given in the figure  \ref{fig2}. Then $ \mathcal{E}(\Phi)> 2\mu(G) $.
\end{lemma}
\begin{figure} [!htb]
\begin{center}
\includegraphics[scale= 0.40]{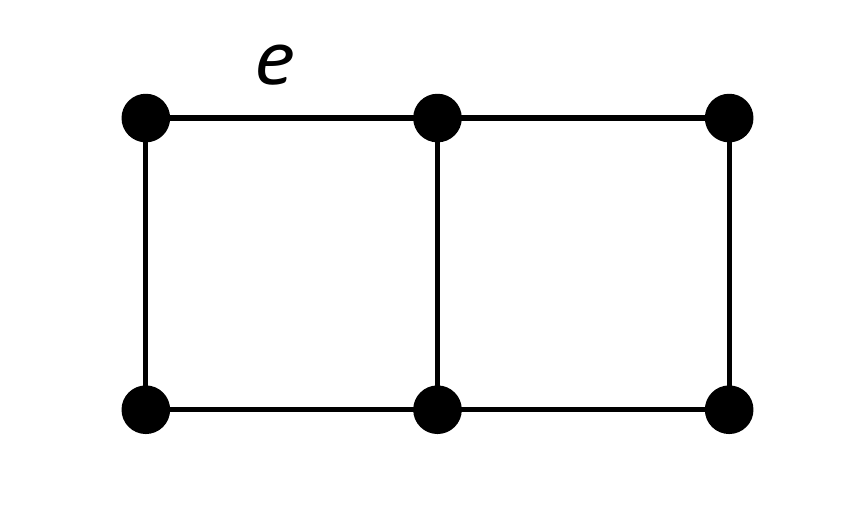}
\caption{  Graph $ G $} \label{fig2}
\end{center}
\end{figure}
\begin{proof}
Let  $ E $ be the cut set consist of the set of edges which are incidence with the edge $ e $ in Figure \ref{fig2}. Then $ G-E= K_2 \oplus P_4 $. Now, by  Lemma \ref{lm6}, $ \mathcal{E}(\Phi) \geq \mathcal{E}(\Phi-E)= \mathcal{E}(( K_2, \varphi))+\mathcal{E}((P_4, \varphi)) $. Since $ ( K_2, \varphi) \sim ( K_2, 1)$ and $(P_4, \varphi) \sim (P_4,1)  $, so by the Lemma \ref{lm5}, we have $ \mathcal{E}(\Phi) > 2+2\mu(P_4)=2\mu(G) $.
\end{proof}

	In the next theorem, we characterize the class of bipartite $ \mathbb{T} $-graphs for which equality holds in Theorem \ref{th4}.
Define  $ N(u)=\{x\in V(G): u\sim x\} $.	
\begin{theorem}\label{th6}
Let $ \Phi=(G, \varphi) $ be any connected  $ \mathbb{T} $-gain bipartite graph with $ n $ vertices. Then $ \mathcal{E}(\Phi)= 2\mu(G) $ if and only if  $ \Phi \sim (K_{\frac{n}{2}, \frac{n}{2}}, 1).$
\end{theorem}
\begin{proof}

 First let us  show that $ G $ is  complete bipartite using induction on the number of vertices.
Let $ |V(G)|=2 $.  If $ \mathcal{E}(\Phi)=2\mu(G) $, then it is clear that $ G=K_{1,1} $. Let us assume that for any connected bipartite $ \mathbb{T} $-gain graph $ (H, \psi) $ with $|V(H)| <n$, if $ \mathcal{E}((H, \psi))=2\mu(H) $, then $ H $ is a complete bipartite graph with same partition size.    Let  $ \Phi=(G, \varphi) $ be any connected bipartite $ \mathbb{T} $-gain graph with $ n $ vertices such that $ \mathcal{E}(\Phi)= 2\mu(G) $.  By the Lemma \ref{lm4}, $ G $ has perfect matching, $ M $ (say). Let $ X $ and $ Y $ be the vertex partition of $ G $ such that $ |X|=|Y|= \frac{n}{2} $.
\vspace*{3pt}
	
\noindent {\bf Claim 1:} For any vertex  $ u\in X $, $ N(u)=Y $.
	
Suppose that  $ N(u)$ is a proper subset of $ Y $. Let $ v^{'} \in Y \setminus N(u) $. Then there exists vertices  $ u^{'}\in X $ and $ v\in Y $ such that the edges $ (u,v)$ and $ (u^{'},v^{'})$ are in  $M$.
	
Let $ P $ be an induced  subgraph formed by the vertices $ \{ u, v, u^{'}, v^{'} \} $. The vertices  $ u^{'}$ and $ v $ are not adjacent in $G$. Suppose they are adjacent. Then $ P $ is isomorphic to $ P_4 $. If $ |V(G)|=4 $, then, by  Lemma \ref{lm5},  
$\mathcal{E}(\Phi)>2\mu(G)  $, a contradiction. Thus $ G=K_{2,2} $.

If $ |V(G)|>4 $, then it is clear that $ \mu(G)=\mu(P)+\mu(G-P) $. By  Lemma \ref{lm5}, $ \mathcal{E}((P, \varphi)) > 2\mu(P) $. Then, by Lemma \ref{lm7}, $ \mathcal{E}(\Phi) > 2\mu(G) $, which is again a contradiction. Thus  $ u^{'} \nsim v $.
	
Let $ Q=(G-P) $. Then $ Q $ is the complementary induced subgraph of $ P $ in $ G $. Therefore, $ \mu(G)=\mu(P)+\mu(Q) =2+\mu(Q)$. Now, we have $ 2\mu(G)=\mathcal{E}(\Phi)\geq \mathcal{E}((P, \varphi))+\mathcal{E}((Q, \varphi))\geq 2(2+\mu(Q))=2\mu(G) $. Thus, $ \mathcal{E}((Q, \varphi))=2\mu(Q) $. Then, by induction hypothesis, $ Q $ is complete bipartite graph with  partition  $ X^{'} $ and $ Y^{'} $ such that $\vert X^{'} \vert = \vert Y^{'} \vert$.  Then  $ X=X^{'}\cup \{u, u^{'} \} $ and $ Y=Y^{'}\cup \{v, v^{'}\} $.
	
\noindent{\bf sub claim: } For every $ x\in X^{'} $ the vertices  $ x$ and $v $ are adjacent,  and for every $ y \in Y^{'} $ the vertices  $ y$ and $u $ are adjacent.
\begin{figure} [!htb]
\begin{center}
\includegraphics[scale= 0.70]{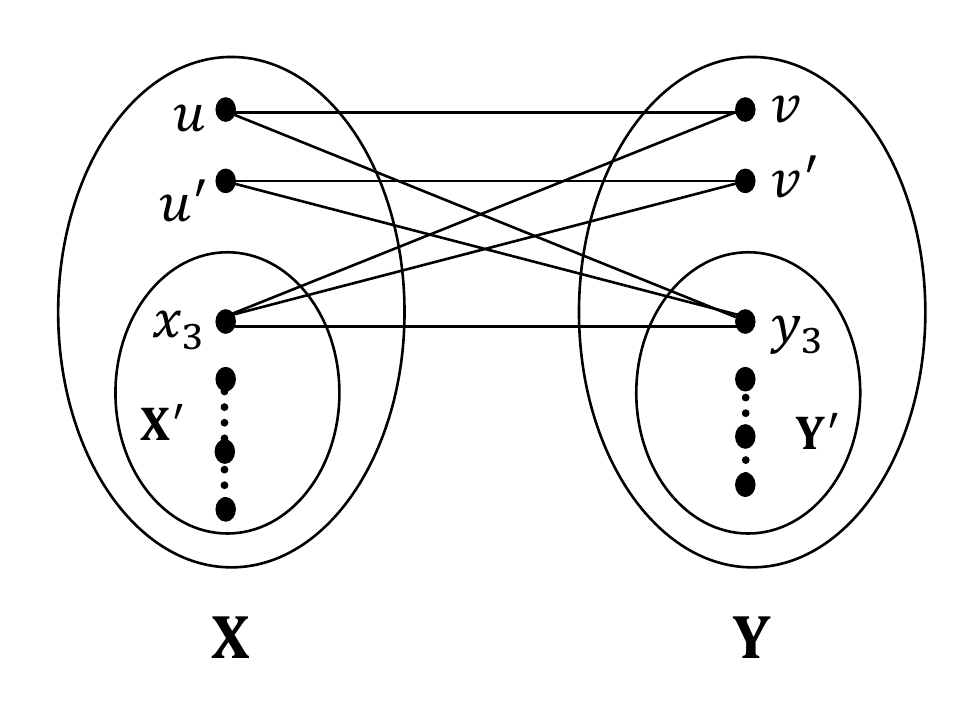}
\caption{  Graph $ G $} \label{fig4}
\end{center}
\end{figure}
	
Since $ G $ is connected, so at least one of the vertices of $u $ or $ v$ is adjacent with the vertices in $ Y^{'} $ or $ X^{'} $, respectively. Without loss of generalities, let us assume that $ u \sim y$ for some $y \in Y^{'} $. Now, every vertex of $ X^{'} $ is adjacent with $ v $. Otherwise, there is a vertex $ x\in X^{'} $ such that $ x \nsim v $. Then the induced underlying subgraph $ H_1 $(say) formed by the vertices $ \{u, v, y,x \} $ is isomorphic to $ P_4 $ and $ \mu(G)=\mu(H_1)+\mu(G-H_1) $. Therefore, by  Lemma \ref{lm7}, $ \mathcal{E}(\Phi) > 2\mu(G) $, a contradiction. Thus every vertex in $ X^{'} $ is adjacent with $ v $.
	
Suppose $ u $ is not adjacent to some of the vertices of $ Y^{'} $. Let $ y_{1}\in Y^{'} $ such that $ y_{1} \nsim u $. Let $ x_{1} \in X^{'} $. Then the induced  subgraph formed by the vertices $ \{y_1, x_1, u, v \} $ is isomorphic to $ P_4 $. Then by an argument  similar to above, we can show that $ \mathcal{E}(\Phi)> 2\mu(G) $. Again we get a contradiction. Therefore, $ X^{'} \subset N(v) $ and $ Y^{'} \subset N(u) $.  Consider $ x_3\in X^{'} $ and $ y_3 \in Y^{'} $. Take the induced  subgraph $ H_2 $ formed by the vertices $ \{ u, v, u^{'}, v^{'}, x_3, y_3\} $ which is given in the Figure \ref{fig4} and  of the form shown in Figure \ref{fig2}. Then $ \mu(G)=\mu(H_2)+\mu(G-H_2) $. Also by the Lemma \ref{lm8}, $ \mathcal{E}((H_2, \varphi))> 2\mu(H_2) $. Therefore, by the Lemma \ref{lm7}, $ \mathcal{E}(\Phi) > 2\mu(G) $. Which is a contradiction. Thus $ N(u)=Y $. Therefore, $ G=K_{\frac{n}{2}, \frac{n}{2}} $.
	
\noindent{\bf Claim 2: }  $ \varphi=1 .$\\
Since $ G=K_{\frac{n}{2}, \frac{n}{2}} $, so $ \mu(G)=\frac{n}{2} $. Then $ \Phi=(K_{\frac{n}{2}, \frac{n}{2}}, \varphi) $ with $ \mathcal{E}(\Phi)=n=\mathcal{E}(K_{\frac{n}{2}, \frac{n}{2}}) $. Therefore, by the Theorem \ref{th5}, $ \Phi\sim(K_{\frac{n}{2}, \frac{n}{2}},1) $.
\end{proof}
A \emph{k-walk} (or simply walk) in an undirected graph $ G $ with vertex set $V(G)=\{v_1, v_2, \dots, v_s\}$ is an alternative sequence of vertices and edges. We simply denote $ v_{i_1}-v_{i_2}-\cdots -v_{i_r} $ as a $ r $-walk from the vertex $ v_{i_1} $ to $ v_{i_r} $, where the vertices and edges in this walk may or may not be distinct. We call a walk  $ v_{i_1}-v_{i_2}-\cdots -v_{i_r}$, a \emph{path} if all the edges in this walk are distinct. If  there is a path in between the  vertices $ v_x $ and $ v_y $, then we call $ v_x $ and $ v_y $ is connected and denoted by $ v_x \leftrightarrow v_y $.

Let $ G_1 $ and $ G_2 $ be two undirected graph with  $V(G_1)= \{v_1, v_2, \dots, v_s\} $ and  $V(G_2)= \{u_1, u_2, \dots, u_t \} $. To avoid the confusion, in definition of  Kronecker product, we use the following notation. If $ v_i \sim v_j $, then the undirected edge in between them is denoted by $ v_iv_j $ and the oriented edge from the vertex $ v_i $ to $ v_j $ is denoted by $ (\overrightarrow{v_iv_j}) $. Let $ \Phi=(G_1, \varphi) $ be any $ \mathbb{T} $-gain graph. A \emph{$ \mathbb{T} $-gain Kronecker product} of $ \Phi $ and a simple graph $ G_2 $ is defined as a $ \mathbb{T} $-gain graph, $ \Phi \otimes G_2=(  G_1 \otimes G_2, \psi) $ on an underlying graph $ G_1 \otimes G_2 $ with vertex set $ V(G_1 \otimes G_2 )=\{(v_p,u_q ): p=1, 2, \dots, s \text{, and } q=1, 2, \dots, t \} $ and edge set $ E( G_1 \otimes G_2)=\{ (v_p,u_q)(v_a,u_b): v_p \sim v_a \text{ and } u_q \sim u_b \}$  such that $ \psi( \overrightarrow{(v_p,u_q)(v_a,u_b)} )=\varphi (\overrightarrow{v_pv_a}) $.  The $\mathbb{T}$-gain graph $ \Phi \otimes K_2 $ is called \emph{$ \mathbb{T} $-gain bipartite double}, where $ K_2 $ is a complete graph of $ 2 $ vertices. We illustrate the following example of a $ \mathbb{T} $-gain bipartite double.
\begin{example}
Let $ G $ be a triangle with vertex set $ V(G)=\{ v_1, v_2, v_3\} $ and $ V(K_2)=\{ x,y\} $. Let $ \Phi=(G, \varphi) $ be a $ \mathbb{T} $-gain graph. Then $ \Phi \otimes K_2 $ is a $ \mathbb{T} $-gain bipartite double. See Figure \ref{fig5}.
\end{example}

\begin{figure} [!htb]
\begin{center}
\includegraphics[scale= 0.95]{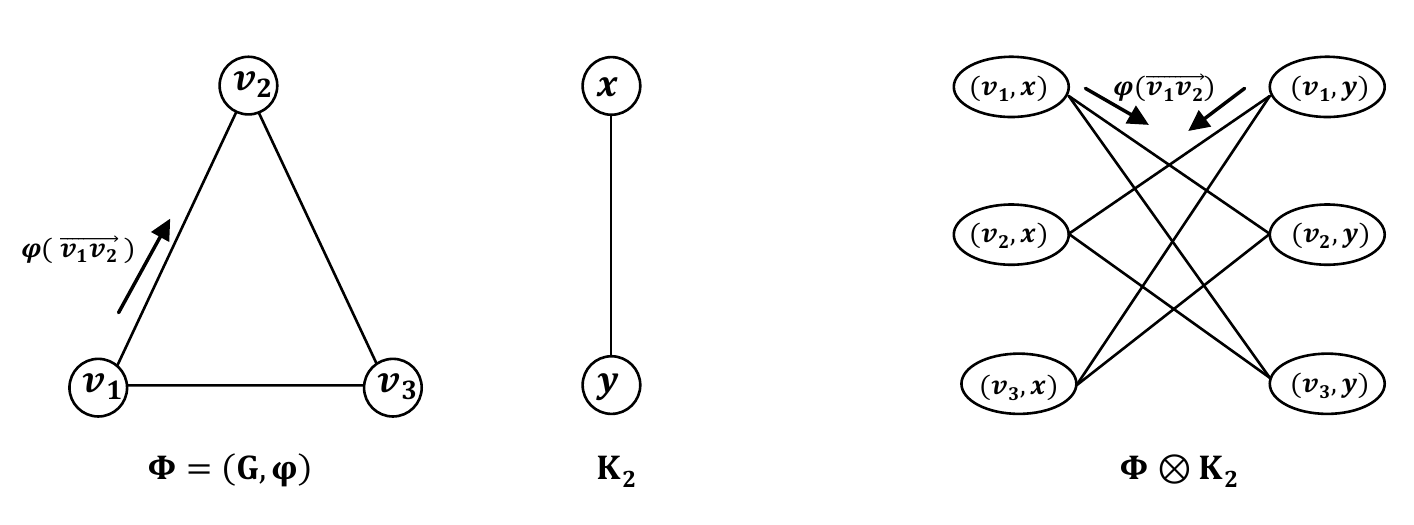}
\caption{ $ \mathbb{T}$-gain bipartite double of $ \Phi $ and $ K_2 $} \label{fig5}
\end{center}
\end{figure}

For any two matrices $ P=(p_{ij})_{r_1 \times r_2} $ and $ Q=(q_{st})_{s_1 \times s_2} $, the Kronecker product of the matrices $P$ and $Q$ are defined as $ P\otimes Q=(p_{ij}Q)_{r_1s_1 \times r_2s_2}$. Now, it is easy to see that  $ A(\Phi \otimes G_2)=A(\Phi) \otimes A(G_2) $.
	
The following lemma is an extension of Lemma \ref{lm.12} for the $\mathbb{T}$-gain graphs.
\begin{lemma} \label{lm10}
Let $ \Phi \otimes G $ be a $ \mathbb{T} $-gain Kronecker product of a $ \mathbb{T} $-gain graph $ \Phi=(G_1, \varphi) $ and an undirected graph $ G $. If $ \spec(\Phi)=\{ \lambda_1, \lambda_2, \dots , \lambda_s\} $ and $ \spec(G)=\{ \gamma_1, \gamma_2, \dots, \gamma_t\} $. Then $ \spec( \Phi \otimes G)=\{\lambda_i\gamma_j: i=1,2, \dots, s, j=1, 2, \dots, t \} $.
\end{lemma}

	  The following lemma is an extension of Lemma \ref{lm.11} for the $\mathbb{T}$-gain graphs.
\begin{lemma} \label{lm15}
If $ \Phi=(G, \varphi) $ be any connected $ \mathbb{T} $-gain graph on a non bipartite graph $ G $, then $ \mathcal{E}(\Phi)> 2\mu(G) $.
\end{lemma}
\begin{proof}
Let $ \Phi=(G, \varphi) $ be a connected $ \mathbb{T} $-gain graph on a non bipartite graph $ G $ with vertex set $ V(G)=\{v_1, v_2, \dots, v_n \} $ and $ m $ edges. If possible let $ \mathcal{E}(\Phi)=2\mu(G) $. Then, by   Lemma \ref{lm4}, $ G $ has perfect matching, say $ M $. Let $v_iv_j  $  denote the edge between the vertices $ v_i $ and $ v_j $, if it exists. Let $ M=\{ v_1v_2, v_3v_4, \dots, v_{n-1}v_n \} $ be a perfect matching in $G$.  Then $ \mu(G)=|M|=\frac{n}{2} $. Therefore, $ \mathcal{E}(\Phi)=2\mu(G)=n $. Let $ V(K_2)=\{ x,y\}$. Now we consider a $ \mathbb{T} $-gain Kronecker product $ \Phi \otimes K_2 $. Then $ |V(\Phi \otimes K_2)|=2n $ and $ |E(\Phi \otimes K_2)|=2m$. It is  easy to see that $ \Phi \otimes K_2 $ is a bipartite graph with a perfect matching $\{(v_1,x)(v_2,y), (v_1,y)(v_2,x), \dots,
 (v_{n-1},x)(v_n,y), (v_{n-1}, y)(v_n,x) \} $. Then $ \mu(\Phi \otimes K_2)=n $. Now, by  Lemma \ref{lm10}, $ \mathcal{E}(\Phi \otimes K_2)=2\mathcal{E}(\Phi)=2n=2\mu(\Phi \otimes K_2)$. That is, $ \mathcal{E}(\Phi \otimes K_2)=2\mu(\Phi \otimes K_2)$. \\
 \textbf{Claim:} $ \Phi \otimes K_2 $ is connected.\\
 The vertex set of $ \Phi\otimes K_2 $ is $ V(\Phi \otimes K_2)=\{ (v_1, x), (v_2, x), \dots , (v_n, x), (v_1, y), (v_2, y), \dots, (v_n,y)\} $. Since $ G $ is connected, so for any pair of vertices $ v_i $ and $ v_j $, there is a path in between them, $ v_{i}=v_{i_0}-v_{i_1}-\cdots-v_{i_t}=v_{j} $, (say). Now $ v_i $ and $ v_j $ is corresponds with the four vertices, $ S=\{(v_i,x), (v_i, y), (v_j,x), (v_j, y) \} $ in $ V(\Phi \otimes K_2) $. We show that any pair of two vertices in that four vertices set is connected. If $ t $ is even, then we have two paths in $ (\Phi\otimes K_2) $, $ (v_{i}, x)=(v_{i_0}, x)-(v_{i_1},y)-\cdots-(v_{i_t},x)=(v_{j},x) $ and $ (v_{i}, y)=(v_{i_0}, y)-(v_{i_1},x)-\cdots-(v_{i_t},y)=(v_{j},y) $. Thus $(v_{i}, x) \leftrightarrow(v_{j},x)$ and $(v_{i}, y) \leftrightarrow(v_{j},y)$. If $ t $ is odd then similarly, $(v_{i}, x) \leftrightarrow(v_{j},y)$ and $(v_{i}, y) \leftrightarrow(v_{j},x)$. Therefore, it is enough to show that $(v_{i}, x) \leftrightarrow(v_{i},y)$. Since $ G $ is connected non bipartite graph, so we can always find a walk from $ v_i $ to $ v_i $ of odd length (walk travels an odd cycle). Then similar to above, $ (v_i,x)\leftrightarrow (v_i,y) $. Therefore, $ (v_i, x) $ is connected with other three vertices of $ S $. Since $ v_i $ and $ v_j $ are arbitrary pair of vertices of $ G $, so any two vertices of $ \Phi\otimes K_2 $ are connected. Thus $ \Phi \otimes K_2 $ is connected.

 Since $ \Phi \otimes K_2  $ is a connected bipartite $ \mathbb{T} $-gain graph of $ 2n $ vertices with $ \mathcal{E}(\Phi \otimes K_2)=2\mu(\Phi \otimes K_2)$, so by Lemma \ref{lm10}, $ \Phi \otimes K_2 \sim (K_{n,n}, 1)  $. Thus $ 2m = |K_{n,n}|=n^{2}$. That is, $ |E(G)|=m=\frac{n^{2}}{2}>\frac{n(n-1)}{2}=|E(K_n)| $. Which is a contradiction. Hence the result.
\end{proof}

\begin{lemma} \label{lm17}
Let $ \Phi=(G, \varphi) $ be any connected $ \mathbb{T} $-gain graph on $ n $ vertices with the matching number $ \mu(G) $. If $ \mathcal{E}(\Phi)=2\mu(G) $, then $ \Phi\sim(K_{\frac{n}{2}, \frac{n}{2}}, 1) $.
\end{lemma}
\begin{proof}
Since $ \mathcal{E}(\Phi)=2\mu(G) $, so by the Lemma \ref{lm15}, $ G $ must be bipartite. Therefore, $ \Phi=(G, \varphi) $ is a connected bipartite $ \mathbb{T} $-gain graph. Now, applying the Theorem \ref{th6} we have $ \Phi \sim (K_{\frac{n}{2}, \frac{n}{2}}, 1) $.
\end{proof}
Let $ A_1, A_2, \dots, A_t $ be $ t $ square complex matrices. Then we denote $ A_1\oplus A_2 \oplus \cdots \oplus A_t $ as a block diagonal matrix with diagonal blocks are $ A_1, A_2, \dots, A_t $. That is, $\bigoplus\limits_{j=1}^{t}A_j=diag(A_1, A_2, \dots, A_t) $.

	In the next theorem, we characterize the class of $ \mathbb{T} $- gain graphs for which equality holds in Theorem \ref{th4}.
\begin{theorem}\label{th4(ii)}
Let $ \Phi=(G, \varphi) $ be any $ \mathbb{T} $-gain graph with matching number $ \mu(G) $. Then $ \mathcal{E}(\Phi)=2\mu(G)$ if and only if  each component of $ \Phi $ is a balanced  complete bipartite $ \mathbb{T} $-gain graph with a perfect matching together with some isolated vertices.
\end{theorem}

\begin{proof}
Let $ G_1,G_2, \dots,G_p, G_{p+1}, \dots, G_{p+r} $ be the connected components of $ G $. Without loss of generality, let us assume that the last $ r $  components are the only isolated vertices. Then $ \mu(G)= \mu(G_1)+\dots+\mu(G_p) $. It is clear that $ A(\Phi)=\bigoplus\limits_{j=1}^{p+r}A((G_j, \varphi)) $
%\textbf{[Did we define this notion?: no Sir.   now I define this at the top of this theorem]}
, So $ \mathcal{E}(\Phi)=\sum\limits_{j=1}^{p}\mathcal{E}((G_j, \varphi)) $. Therefore, by the Theorem \ref{th4}, we have,
\begin{equation}
2\mu(G)=\mathcal{E}(\Phi)=\sum\limits_{j=1}^{p}\mathcal{E}((G_j, \varphi)) \geq 2\sum\limits_{j=1}^{p}\mu(G_j)=2\mu(G)
\end{equation}
Thus $ \mathcal{E}((G_j, \varphi))=2\mu(G_j)$, for each $ j=1,2, \dots , s $. Now, using the Lemma \ref{lm17}, we can derive the result.
\end{proof}

As an application of the above theorem, we can establish a relationship among the energy of $ \mathbb{T} $-gain graph, the  vertex cover number and the number of odd cycles.  This result generalizes one of the main results of  \cite{Wei-Li}. Let $ \Phi=(G, \varphi) $ be a $ \mathbb{T} $-gain graph with vertex set $ V(G)$. Let $ u\in  V(G) $. Then $ (\Phi-u) $ denotes an induced subgraph of $ \Phi $ with vertex set $ V(G)\setminus\{u\} $.

\begin{theorem} \label{th8}
Let $ \Phi=(G, \varphi) $ be any $ \mathbb{T} $-gain graph on $ G $ with $ c(G) $ number of odd cycles and  vertex cover number $ \tau(G) $. Then
\begin{equation*}
\mathcal{E}(\Phi) \geq 2\tau(G)-2c(G).
\end{equation*}
Equality occurs if and only if each component  of $ \Phi$ is a balanced complete bipartite $ \mathbb{T} $-gain graph with a perfect matching together with some isolated vertices.
\end{theorem}

\begin{proof}
Let $ \Phi=(G, \varphi) $ be any $ \mathbb{T} $-gain graph with $ c(G) $ number of odd cycles. Let us prove the bound using  induction on the number of odd cycles $ c(G) $. If $ c(G)=0 $, then $ G $ is bipartite. Therefore $ \mu(G)=\tau(G) $. Now, by  Theorem \ref{th4}, we have $ \mathcal{E}(\Phi)\geq 2\mu(G)=2\tau(G)-2c(G) $.  Assume that the statement is true for any  $ \mathbb{T} $-gain graph  with the number of odd cycles is at most $ (c(G)-1) $. Consider $ \Phi $ with $ c(G) \geq 1 $ number of odd cycles. Let $u $ be a vertex in an odd cycle of $ G $. Then the number of odd cycles, say $ c^{'} $, of $ \Phi-u $ is at most $ (c(G)-1) $. Thus, by induction hypothesis, $ \mathcal{E}(\Phi-u)\geq 2\tau(G-u)-2c^{'} $. Since $ u $ is an isolated vertex, so, by  Lemma \ref{lm13}, $ \mathcal{E}(\Phi)>\mathcal{E}(\Phi-u)$.

It is easy to see that $\tau(G-u) \geq \tau(G)-1 $. Therefore, $ \mathcal{E}(\Phi)>\mathcal{E}(\Phi-u)\geq 2\tau(G)-2c(G) $.

Now, let  $ \mathcal{E}(\Phi)=2\tau(G)-2c(G) $. If $ c(G) \geq 1 $, then, by the above observation, $ \mathcal{E}(\Phi)> 2\tau(G)-2c(G) $. which is a contradiction. That is $ c(G)=0 $. Therefore, $ G $ is bipartite and $ \mu(G)=\tau(G) $. Thus $ \mathcal{E}(\Phi)=2\mu(G) $. Now, by  Theorem \ref{th6}, $ \Phi $ is the disjoint union of some balanced complete bipartite $ \mathbb{T} $-gain graphs with a perfect matching together with some isolated vertices.
\end{proof}
%==============================================
	\section {Upper bound of energy of $ \mathbb{T} $-gain graph in terms of vertex cover number and largest vertex degree}\label{upper-bound}
	
		In this section, our main objective is to obtain an upper bound for the energy of a  $ \mathbb{T} $-gain graph in terms of the vertex cover number and the largest vertex degree. This result is the counter part of the corresponding known result about undirected graph [Theorem \ref{th-1.3}] and mixed graph [Theorem \ref{th-1.4}]. Furthermore, we characterize all $\mathbb{T}$-gain graphs for which the upper bound is attained. This characterization completely solve one of the open problem  \cite{Wei-Li}.

	\begin{theorem} \label{th-5.3}
	   		Let $ \Phi=(G, \varphi) $ be any $ \mathbb{T} $-gain graph
	   		with the vertex cover number $ \tau(G) $, and maximum vertex degree $
	   		\Delta (G)$. Then,
	   		\begin{equation}\label{2-eq.2}
	   		\mathcal{E}(\Phi)\leq 2 \tau(G) \sqrt{\Delta(G)}.
	   		\end{equation}
	\end{theorem}	   		
	\begin{proof}
	Let $ \Phi=(G, \varphi) $ be any $ \mathbb{T} $-gain graph with vertex cover number $ \tau(G) $.   We prove the result by  induction on $ \tau(G) $. If $ \tau(G)=1$, then $ G$ must be $ K_{1,r} $, for some $ r $ together with some isolated vertices. Therefore, $ \Phi $  is balanced. Now $ \mathcal{E}(\Phi)=\mathcal{E}(K_{1, r})=2\sqrt{r}=2\tau(G)\sqrt{\Delta(K_{1,r})} $.
	
	Let us assume that for any $ \mathbb{T} $-gain graph $\Psi= (G_1, \psi ) $ with $ \tau(G_1)< \tau(G) $, we have $ \mathcal{E}(\Psi)\leq 2\tau(G_1)\sqrt{\Delta(G_1)} $. Let $ U $ be a minimum vertex cover of $ G $. Then $ |U|=\tau(G) \geq 2$. Let $ x \in U $. Let $ S $ be an induced subgraph of $ G $ which is formed by removing the vertex $ x $, and the edges incident with $x $ from $ G $. That is $ S=G-x $. Then $ \tau(S)=\tau(G)-1 $. Therefore, by the induction hypothesis, $ \mathcal{E}(\Phi-x)=\mathcal{E}((S, \varphi))\leq 2 \tau(S)\sqrt{\Delta(S)}$. After a suitable relabeling of vertices, we can express $ A(\Phi) $ as
	\begin{equation*}
	A(\Phi)=\left[ \begin{array}{ccc}
	0& {\textbf v^{*}} & {\textbf 0} \\ {\textbf v} & A_1 & Y^{*}\\ {\textbf 0}& Y & A_2 \end{array} \right]=\left[ \begin{array}{ccc}
	0& {\textbf v^{*}} & {\textbf 0} \\ {\textbf v} & {\textbf 0} & {\textbf 0}\\ {\textbf 0}& {\textbf 0} & {\textbf 0} \end{array} \right]+\left[ \begin{array}{ccc}
	0& {\textbf 0} & {\textbf 0} \\ {\textbf 0} & A_1 & Y^{*}\\ {\textbf 0}& Y & A_2 \end{array} \right]
	\end{equation*}
	Here the first column and the first row are associated with the vertex $ x $. Let the degree of $ x $ be $ d $. Then $\left[ \begin{array}{cc}
	0& {\textbf v^{*}}  \\ {\textbf v} & {\textbf 0} \end{array} \right]  $ and $ \left[ \begin{array}{cc}
	 A_1 & Y^{*}\\ Y & A_2 \end{array} \right] $  are the adjacency matrices of the $ \mathbb{T} $-gain subgraphs $ (K_{1, d}, \varphi) $ and $ (S, \varphi) $, respectively. By Theorem \ref{lm2.15}, %\textbf{[Include the theorem in section 2, and refer here: Yes sir I include.]}
	  we have
	 \begin{equation}\label{2-eq.4}
	 \mathcal{E}(\Phi)\leq \mathcal{E}(K_{1, d})+\mathcal{E}((S, \varphi)) \leq 2\sqrt{d}+2 \tau(S)\sqrt{\Delta(S)}\leq 2\tau(G)\sqrt{\Delta(G)}.
	 \end{equation}
	 \end{proof}

    \begin{theorem} \label{th-5.4}
		Let $ \Phi=(G, \varphi) $ be any $ \mathbb{T} $-gain graph on $ G $
		with  vertex cover number $ \tau(G) $
	 and maximum vertex degree $
		\Delta(G) $. Then
		\begin{eqnarray}
		\mathcal{E}(\Phi)= 2 \tau(G) \sqrt{\Delta(G)}
		\end{eqnarray}
		if and only if $ \Phi $ is the disjoint union of $ \tau(G) $ copies of
		balanced  $ \mathbb{T} $-gain graph $ (K_{1, \Delta(G)}, 1) $ together
		with some isolated vertices.
   \end{theorem}
		  		  	
	 \begin{proof}

	 First let us show that all the vertices of $ U $ have the same vertex degree, $ \Delta(G).$ % \noindent{\bf Claim 1:} All the vertices of $ U $ have same vertex degree, $ \Delta(G) $.\\
	Let $ x \in U $ be any vertex in $ U $ (as in Theorem \ref{th-5.3}). Since $ \mathcal{E}(\Phi)=2\tau(G)\sqrt{\Delta(G)} $, so all the inequalities of (\ref{2-eq.4})  become equations. So $ \mathcal{E}((S, \varphi)) =2\tau(S)\sqrt{\Delta(S)}$, and $ d=\Delta(S)=\Delta(G) $. As $ x $ is arbitrary, so all the vertices of $ U $ are of degree $ \Delta(G) $. \\
	 Now we claim that the underlying graph $ G $ is bipartite. % \noindent Claim 2:  $ G $ is a bipartite graph.\\
	 Let $ W=V(G)\setminus U $. It is clear that $ U\setminus \{x\} $ is a minimum vertex cover of the induced subgraph $ S $. Also, we have $ \mathcal{E}((S, \varphi)) =2\tau(S)\sqrt{\Delta(S)}$. Now, applying the argument
	
	 to $ S $. Therefore, all the vertices of $ U \setminus \{x\} $ in $ S $  %\textbf{[This is S, right? I hope there is no H in this theorem.: Yes Sir. That is S]}
	 is of degree $ \Delta(S) $. Also we know that $ \Delta(S)=\Delta (G) $. Since $ d=\Delta(G) $, so there is no edge  between the vertex $ x $ and the vertices of $ U \setminus \{x\} $. As $x$ is arbitrary, so we get no two vertices of $ U $ are adjacent.  Now $ U $ is a minimum vertex cover of $ G $, so no two vertices of $ W $ are adjacent. Hence $ G $ is a bipartite graph with vertex partition sets $ U $ and $ W $.

	 Let $ G_1, G_2, \dots, G_p $ be the only nontrivial components of $ G $ (That is components contain at least one edge). Then,
	 $$ 2\tau(G)\sqrt{\Delta(G)}=\mathcal{E}(\Phi)=\sum\limits_{j=1}^{p}\mathcal{E}((G_j, \varphi))\leq\sum\limits_{j=1}^{p}2\tau(G_j)\sqrt{\Delta(G_j)}\leq 2\tau(G)\sqrt{\Delta(G)}.$$
	 From the above expression, we get $ \mathcal{E}((G_j, \varphi)) = 2\tau(G_j)\sqrt{\Delta(G_j)}$ and $ \Delta(G_j)=\Delta(G) $, for $ j=1,2, \dots, p $.
	
	 Now let us show that the rank of each component $ (G_j, \varphi) $ is $ 2 $. 
	  Let $ r_j $ be the rank of $(G_j, \varphi)  $. Since $ (G_j, \varphi) $ is bipartite, so its  spectrum is symmetric with respect to origin. Thus $ r_j$ is an even number and $r_j \geq 2 $. Let $ \lambda_1 \geq \lambda_2\geq \cdots \geq \lambda_{r_j} $ be the nonzero eigenvalues of $ (G_j, \varphi).$ Suppose that $ r_j >2 $. Then $ \lambda_1 > \lambda_2>0 $. Therefore, by the Cauchy-Schwartz inequality
	  \[ \mathcal{E}((G_j, \varphi))=\sum\limits_{t=1}^{j}|\lambda_t|<\sqrt{r_j}\sqrt{\sum\limits_{t=1}^{j}\lambda_t^{2}} =\sqrt{2|E(G_j)|r_j}.\]
	  For any $ \mathbb{T} $-gain graph $ \Psi=(B, \psi) $ on a bipartite graph $ B $, we know that $ |E(B)|\leq\tau(B)\Delta(B) $. By Lemma \ref{2-lm5}, we have $ \rank(\Psi)\leq 2\mu(B)=2\tau(B)$. Hence $ \mathcal{E}((G_j, \varphi))< 2\tau(G_j)\sqrt{\Delta(G_j)} $, a contradiction (as for each component, $ \mathcal{E}((G_j, \varphi))= 2\tau(G_j)\sqrt{\Delta(G_j)} $). Hence the rank of $ (G_j, \varphi)$ is $ 2 $ for $ j=1,2, \cdots, p $.
	
	  Since each nontrivial component $ (G_j, \varphi) $ is bipartite and of rank $ 2 $. Now $ (G_j, \varphi) $ is of rank $ 2 $ if and only if it has exactly one positive eigenvalue. Therefore, by Lemma \ref{lem5}, $ (G_j, \varphi)\sim (K_{a,b}, 1) $.  Without loss of generality, consider $ a\leq b $. Then $ \tau(G_j)=a $ and $ \Delta(G_j)=b $. Now $ \mathcal{E}((G_j, \varphi))=2\tau(G_j)\sqrt{\Delta(G_j)}=2\tau(G_j)\sqrt{\Delta(G_j)}=2a\sqrt{b} $. On the other hand $ \mathcal{E}((G_j, \varphi))= \mathcal{E}((K_{a,b},1))=2\sqrt{ab} $. Thus $2a\sqrt{b}=2\sqrt{ab}  $. Thus $ a=1 $, and hence $ b=\Delta(G_j)=\Delta(G)  $. Therefore, for each $ j=1,2, \dots, p $, $ (G_j, \varphi) \sim (K_{1, \Delta(G)},1) $.
	  \end{proof}

%============================================================

\section*{Acknowledgments}
	Aniruddha Samanta thanks University Grants Commission(UGC)  for the financial support in the form of the Senior Research Fellowship (Ref.No:  19/06/2016(i)EU-V; Roll No. 423206). M. Rajesh Kannan would like to thank the SERB, Department of Science and Technology, India, for financial support through the projects MATRICS (MTR/2018/000986) and Early Career Research Award (ECR/2017/000643).
	
	\bibliographystyle{amsplain}
	\bibliography{raj-ani-ref1}

\end{document}